\documentclass[12pt,reqno]{amsart}

\usepackage{amsfonts}
\usepackage{amsmath}
\usepackage{amssymb}
\usepackage{thmtools}
\usepackage{amsthm}
\usepackage[protrusion=true,expansion=false]{microtype}
\usepackage{enumitem}
\usepackage{tikz, tikz-cd}
\usepackage{ytableau}
\usepackage{quiver}
\usepackage[colorlinks=true,linkcolor=blue,citecolor=blue,urlcolor=blue]{hyperref}

\theoremstyle{plain}
\newtheorem{theorem}{Theorem}[section]
\newtheorem{proposition}[theorem]{Proposition}
\newtheorem{lemma}[theorem]{Lemma}
\newtheorem{corollary}[theorem]{Corollary}
\newtheorem{conjecture}[theorem]{Conjecture}

\theoremstyle{definition}
\newtheorem{definition}[theorem]{Definition}
\newtheorem{example}[theorem]{Example}

\theoremstyle{remark}
\newtheorem{remark}[theorem]{Remark}

\numberwithin{equation}{section}

\usepackage[]{cite}

\newcommand\CC{\mathbb{C}}
\newcommand\QQ{\mathbb{Q}}
\newcommand\ZZ{\mathbb{Z}}

\newcommand{\cP}{\mathcal{P}}
\newcommand{\cR}{\mathcal{R}}
\newcommand{\cL}{\mathcal{L}}

\newcommand\Par{\mathrm{Par}}
\newcommand\arm{\mathrm{arm}}
\newcommand\coarm{\mathrm{coarm}}
\newcommand\leg{\mathrm{leg}}
\newcommand\coleg{\mathrm{coleg}}
\newcommand\Leg{\mathrm{Leg}}
\newcommand\SSYT{\mathrm{SSYT}}
\newcommand{\Exp}{\mathrm{Exp}}
\newcommand{\ch}{\mathrm{ch}}
\newcommand{\Left}{\mathrm{Left}}
\newcommand{\maj}{\mathrm{maj}}
\newcommand{\qbinom}[2]{\genfrac{[}{]}{0pt}{}{#1}{#2}}

\title[Hall--Littlewood $t$--$q$ transition coefficients]{Changing parameters in Hall-Littlewood functions}
\author{Aritra Bhattacharya}
\address{School of Mathematics and Statistics, Beijing Institute of Technology, Beijing 102400, China} 
\email{matharitra@gmail.com}

\author{Ziyi Dai}
\address{Department of Mathematics, Peking University, Beijing, 100871, China.}
\email{1909506735@qq.com}

\author{Qilin Hou}
\address{Department of Mathematics, Southern University of Science and Technology, Shenzhen, 518055, China.}
\email{houql2023@mail.sustech.edu.cn}

\author{Edward Tong}
\address{Department of Mathematics, University of Michigan, Ann Arbor, MI 48109, USA}
\email{edtong@umich.edu}

\date{\today}

\begin{document}
	
	\begin{abstract}
		We study the transition coefficients $b_{\lambda,\mu}(q,t)$ of the Hall--Littlewood $P$-functions $P_\lambda(t)$ expanded in the basis $P_\mu(q)$. We find a tableaux formula when $\lambda$ is a hook shape which generalizes term-by-term Lascoux-Schutzenberger's charge formula as well as Macdonald's monomial expansion formula for Hall-Littlewood functions. We conjecture that $b_{\lambda,\mu}(t+1,t)\in\mathbb{Z}_{\ge0}[t]$ and prove it when $\lambda$ is a hook shape, when $\lambda$ has two-columns and some more simple cases. We prove that $b_{\lambda,\mu}(t+n,n)\in\mathbb{Z}_{\ge0}[t]$ for every $n\in\mathbb{Z}_{\ge0}$, extending work of Alexandersson, and we record a stability property and some reduction principles. 
	\end{abstract}
	
	
	\maketitle
	
	\section{Introduction}\label{sec:intro}
	
	The Hall--Littlewood $P$-functions $P_\lambda(u)$ are a one-parameter family of symmetric functions depending on a parameter $u$ and indexed by integer partitions $\lambda$. These functions interpolate between two important bases of the space of symmetric functions: the Schur functions $s_\lambda = P_\lambda(0)$ and the monomial symmetric functions $m_\lambda = P_\lambda(1)$.
	
	The transition matrix coefficients expressing the Schur functions in terms of the monomial symmetric functions are the Kostka numbers $K_{\lambda,\mu}$,
	\begin{align}\label{eq:stom}
		s_\lambda = \sum_{\mu}^{} K_{\lambda,\mu} m_\mu.
	\end{align}
	This equation is a direct shadow of the weight space decomposition of the irreducible $GL_n(\CC)$ representations (for large enough $n$), and Kostka numbers are the dimensions of the weight spaces of irreducible $GL_n(\CC)$-representations. There are many combinatorial models which enumerate these Kostka numbers, one of the most popular being the semistandard Young tableaux: $K_{\lambda,\mu}$ is the number of semistandard Young tableaux of shape $\lambda$ and content $\mu$.
	
	A $u$-analogue of the above picture is obtained by considering the transition matrix coefficients from the Schur functions expanded in the Hall--Littlewood $P$-basis; the coefficients $K_{\lambda,\mu}(u)$ are called the Kostka--Foulkes polynomials.
	\begin{align}\label{eq:stoP}
		s_\lambda = \sum_{\mu}^{} K_{\lambda,\mu}(u) P_\mu(u).
	\end{align}
	Under the Satake isomorphism this equation is a shadow of the change of basis between the canonical and standard bases of the spherical Hecke algebra, and the polynomials $K_{\lambda,\mu}(u)$ are certain Kazhdan--Lusztig polynomials, encoding the Poincar\'e polynomials of certain intersection cohomology. From this standpoint $K_{\lambda,\mu}(u) \in \ZZ_{\geq 0}[u]$, polynomials with positive integral coefficients. Lascoux and Sch\"utzenberger developed a subtle statistic on semistandard Young tableaux, called charge and denoted $\ch$ in this paper, which serves to give a combinatorial interpretation of these coefficients:
	\begin{align}
		K_{\lambda,\mu}(u) = \sum_{T \in \SSYT(\lambda,\mu)}^{} u^{\ch(T)},
	\end{align} where $\SSYT(\lambda,\mu)$ denotes the set of semistandard Young tableaux of shape $\lambda$ and content $\mu$.
	
	In another direction, one can generalize \eqref{eq:stom} by changing the Schur function side to Hall--Littlewood $P$-functions; in this case Macdonald \cite{MacMainBook} has an expansion formula by summing over semistandard Young tableaux again:
	\begin{align}
		P_\lambda(u) = \sum_{\mu}^{} \bigg(\sum_{T \in \SSYT(\lambda,\mu)}^{} \psi_T(u)\bigg) m_\mu,
	\end{align} where $\psi_T(u) \in \ZZ[u]$ is explicitly described in \S\ref{sec:monexpform}.
	
	\subsection{The transition coefficients \texorpdfstring{$b_{\lambda,\mu}(q,t)$}{b(q,t)}}
	
	In view of the above it is desirable to study the transition matrix between two families of Hall--Littlewood $P$-functions, parametrized by different parameters. More specifically, consider $\Lambda = \Lambda_{\QQ(q,t)}$, the space of symmetric functions with coefficients from $\QQ(q,t)$. The families $(P_\lambda(t))_{\lambda \in \Par}$ and $(P_\lambda(q))_{\lambda \in \Par}$ are both $\QQ(q,t)$-bases of $\Lambda$, where $\Par$ denotes the set of integer partitions. We want to investigate properties of $b_{\lambda,\mu}(q,t)$ defined by
	\begin{align}
		P_\lambda(t) = \sum_{\mu \in \Par}^{} b_{\lambda,\mu}(q,t) P_\mu(q).
	\end{align}
	Writing the matrices $b(q,t) = (b_{\lambda,\mu}(q,t))_{\lambda,\mu \in \Par}$, $K(u) = (K_{\lambda,\mu}(u))_{\lambda,\mu \in \Par}$, we get
	\begin{align}
		b(q,t) = K(t)^{-1}K(q).
	\end{align} Carbonara \cite{Carbonara_inverseKFp} obtained a combinatorial formula for the entries of the inverse Kostka--Foulkes matrix, thus in principle we have a signed formula for $b_{\lambda,\mu}(q,t)$.
	
	By the above discussion
	\begin{align}
		b_{\lambda,\mu}(q,0) = \sum_{T \in \SSYT(\lambda,\mu)}^{} q^{\ch(T)} \qquad \hbox{ and } \qquad b_{\lambda,\mu}(1,t) = \sum_{T \in \SSYT(\lambda,\mu)}^{} \psi_T(t).
	\end{align}
	\subsection{A tableaux formula for hook shapes}
	
	It would be very interesting if one could find a formula that attaches a weight to tableaux in such a way that at $t=0$ and $q=1$ the weights specialize to the charge and Macdonald's monomial expansion weights term-by-term, respectively. We answer this question in the positive for the case when $\lambda$ is a hook shape and some more simple cases. We now describe our result for hook shape.
	
	We defer the definition of charge to \S\ref{sec:charge}; for now we only comment that given a tableau $T \in \SSYT(\lambda,\mu)$, for some partitions $\lambda,\mu$, each cell $x$ in the Young diagram of $\lambda$ gets assigned a weight $\ch_T(x)$ such that $\ch(T) = \sum_{x \in \lambda}^{} \ch_T(x)$.
	
	Suppose $\lambda = (n-k,1^k)$, $n-k > 1, k \geq 0$. For $T \in \SSYT(\lambda,\mu)$ we define $\psi_{T}(q,t) \in \ZZ[q,t]$ below.
	Let $$m = \begin{cases}
		T(1,2)-1, &\hbox{ if } T(1,2) \neq 1,
		\\
		\qquad 1, &\hbox{ otherwise.}
	\end{cases}$$
	Define \begin{align}
		\psi_{T}(x;q,t) = \begin{cases}
			(q^{\ch_T(x)}-t), &\hbox{ if } \coarm(x) > 0, \hbox{ and } T(\Left(x)) < T(x),
			\\
			\qquad q^{\ch_T(x)}, &\hbox{ otherwise},
		\end{cases}
	\end{align} where $\Left(x)$ denotes the box immediately on the left of $x$ in the same row.
	Let
	\begin{align}
		\psi_{T}(q,t) = [m]_t \cdot \prod_{x \in \lambda}^{} \psi_{T}(x;q,t),
	\end{align}
	where $[m]_t = 1+t+\ldots+t^{m-1}$ for $m \in \ZZ_{>0}$.
	Then $$\psi_{T}(q,0) = q^{\ch(T)}  \hbox{ and } \psi_{T}(1,t) = \psi_T(t).$$
For the unique tableau of shape $(1^n)$ let  $\psi_T(q,t) = 1$. Then we prove the following theorem.
	
	\begin{theorem}\label{thm:main}
		For $\lambda = (n-k,1^k)$, $n > k \geq 0$, and $\mu \in \Par$,
		\begin{align}\label{eq:thmmaineq}
			b_{\lambda,\mu}(q,t) = \sum_{T \in \SSYT(\lambda,\mu)}^{} \psi_{T}(q,t).
		\end{align}
	\end{theorem}
	The proof of \autoref{thm:main} is given in \S\ref{sec:hook}. 
	When $\lambda$ is a single row, \eqref{eq:thmmaineq} simplifies to the following product formula (proved in \autoref{th:b_singlerow}):
	\begin{equation}\label{eq:b_singlerow_intro}
		b_{(n),\mu}(q,t) = q^{n(\mu)-\binom{\ell}{2}} \prod_{i=1}^{\ell-1}(q^i-t),
	\end{equation}
	where $\ell = \ell(\mu)$ and $n(\mu) = \sum_{i \geq 1}^{} (i-1)\mu_i$.
	
	In the final two sections \S\ref{sec:smallcases} and \S\ref{sec:nk1} we show that analogous formulas work when $\mu$ is of the form $(n,k,1)$ and small hooks $(n-2,1^2)$,$(n-3,1^3)$.
	
	\subsection{A positivity conjecture and further results}
	
	\autoref{thm:main} immediately implies that in the case that $\lambda$ is of hook shape, we have $b_{\lambda,\mu}(t+1,t) \in \ZZ_{\geq 0}[t]$. We conjecture that this is true for general $\lambda,\mu$, and have checked it for partitions of size at most $26$ in \texttt{SageMath}.
	\begin{conjecture}\label{conj:B}
		For $\lambda,\mu \in \Par$, $b_{\lambda,\mu}(t+1,t) \in \ZZ_{\geq 0}[t]$.
	\end{conjecture}
	At $t=0$, this is the positivity of Schur functions in the monomial basis. We note that there is at least a partial positivity even when $t=-1$: when $\lambda$ is a strict partition, i.e., no two parts are equal, then $P_\lambda(-1)$ expands positively in the Schur basis \cite[Ch. III, (8.17)]{MacMainBook}. The functions $P_\lambda(-1)$ are called Schur $P$-functions and they play an important role in the projective representation theory of symmetric groups. In a MathOverflow post \cite{MO190646} I. Makhlin conjectured that $P_\lambda(-t)$ is monomial positive when $\lambda_i - \lambda_{i+1} \geq 2$ for all $i<\ell(\lambda)$. R. Stanley conjectures in a response in that same post that in fact under those conditions $P_\lambda(-t)$ is Schur positive. A discussion of both our conjecture and Makhlin's conjecture can be found in the \href{https://www.symmetricfunctions.com/hallLittlewood.htm#hallLittlewoodP}{The symmetric functions catalogue} maintained by Alexandersson.
	
	In this paper we prove \autoref{conj:B} in a few more cases.
	\begin{theorem}\label{thm:twocol-intro}
		\autoref{conj:B} holds when $\lambda_1 \leq 2$, i.e., when $\lambda$ is a two-column shape.
	\end{theorem}
	The proof is given in \S\ref{sec:twocol}: an explicit alternating formula involving $q$-binomial coefficients for $b_{(2^a1^b),(1^N)}(q,t)$ (\autoref{lem:twocolumn}) is paired up term-by-term and shown to be positive at $q=t+1$. We have not found a tableaux formula in this case.
	
	\subsection{Positivity for $(q,t)=(t+n,n)$}
	
	Using the Hall scalar product, which is characterized by the Schur functions being orthonormal, taking dual bases gives
	\begin{align}
		Q'_\mu(q) = \sum_{\lambda}^{} b_{\lambda,\mu}(q,t) Q'_\lambda(t),
	\end{align}
	where $Q'_\lambda(u)$ denotes the transformed Hall--Littlewood function.
	
	From this perspective, Alexandersson \cite{Alexandersson_LLTelementarymelting} proved that $b_{\lambda,\mu}(t+1,1) \in \ZZ_{\geq 0}[t]$, or equivalently, $Q'_\mu(t+1)$ is $h$-positive, where $h$ denotes the complete homogeneous symmetric function basis. After applying the $\omega$ involution, these are a part of a family of symmetric functions known as vertical strip LLT functions, and Alexandersson's result also extends to this family. We mimic the argument from \cite{Alexandersson_LLTelementarymelting} to prove 
	\begin{theorem}\label{thm:tplusn}
		Let $\lambda,\mu \in \Par$ and $n \in \ZZ_{\geq 0}$. Then $b_{\lambda,\mu}(t+n,n) \in \ZZ_{\geq 0}[t].$
	\end{theorem}
	
	In \cite{AlexanderssonSulzgruber} a combinatorial formula for $b_{\lambda,\mu}(t+1,1)$ is given by Alexandersson and Sulzgruber; it would be interesting if it can be generalized to $b_{\lambda,\mu}(t+n,n)$.
	
	Finally, we give a formula for $Q'_{(n,k,1)}(q)$ in the $Q'(t)$-basis in \autoref{thm:nk1}, and this gives tableaux formulas analogous to the hook case for $b_{\lambda,(n,k,1)}(q,t)$.
	
	\subsection{A positivity philosophy}\label{sec:macsquare}
	
	This subsection is only for some more motivation regarding \autoref{conj:B}.
	The Hall--Littlewood $P$-functions have a two-parameter generalization developed by Macdonald, the Macdonald $P$-functions $(P_\lambda(q,t))_{\lambda \in \Par}$, which form a basis of $\Lambda_{\QQ(q,t)}$. Thinking of $q$ and $t$ as real variables in the co-ordinate plane with $x=q,y=t$, we can picture specializations of Macdonald polynomials in the Macdonald square, pictured in \autoref{fig:macsquare}, which records the specializations
	\begin{align*}
		P_\lambda(0,t) = P_\lambda(t), \qquad P_\lambda(q,0) = W_\lambda(q), \qquad P_\lambda(1,t) = e_{\lambda'} \qquad \hbox{ and } \qquad P_\lambda(q,1) = m_\lambda,
	\end{align*}
	where $W_\lambda(q) = \omega(Q'_{\lambda'}(q))$ are the $q$-Whittaker functions. 
	
	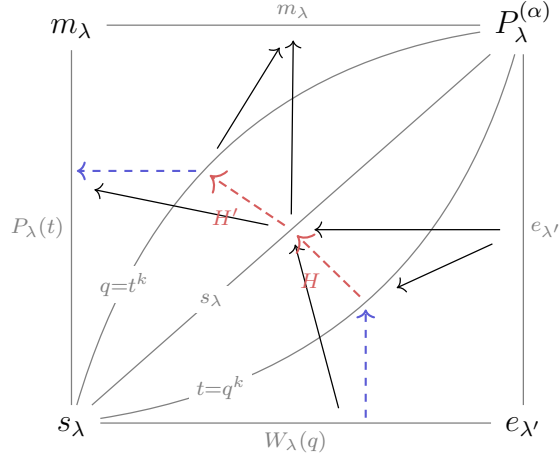
\begin{figure}[h]
		\centering
		\begin{tikzcd}[row sep=2.25em, execute at end picture={
				\draw[color={rgb,255:red,92;green,92;blue,214}, dashed, thick, ->, shorten <=2pt, shorten >=2pt] (2) -- (2 -| 4);
				\draw[color={rgb,255:red,92;green,92;blue,214}, dashed, thick, ->, shorten <=2pt, shorten >=2pt] (1 |- 5) -- (1);
			}]
			{\tiny m_{\lambda}} &&&&& {\tiny P^{(\alpha)}_\lambda} \\
			\\
			\\
			\\
			\\
			{\tiny s_\lambda} &&&&& {\tiny e_{\lambda'}}
			\arrow[""{name=0, anchor=center, inner sep=0}, "{\tiny s_\lambda}"{description, pos=0.3}, color={rgb,255:red,128;green,128;blue,128}, no head, from=6-1, to=1-6]
			\arrow[""{name=1, anchor=center, inner sep=0}, "{\tiny t=q^k}"{description, pos=0.2}, color={rgb,255:red,128;green,128;blue,128}, curve={height=50pt}, no head, from=6-1, to=1-6]
			\arrow[""{name=2, anchor=center, inner sep=0, pos=0.44}, "{\tiny q=t^k}"{description, pos=0.2}, color={rgb,255:red,128;green,128;blue,128}, curve={height=-50pt}, no head, from=6-1, to=1-6]
			\arrow[""{name=3, anchor=center, inner sep=0}, "{\tiny m_\lambda}", color={rgb,255:red,128;green,128;blue,128}, no head, from=1-1, to=1-6]
			\arrow[""{name=4, anchor=center, inner sep=0, pos=0.38}, "{\tiny P_\lambda(t)}"', color={rgb,255:red,128;green,128;blue,128}, no head, from=1-1, to=6-1]
			\arrow[""{name=5, anchor=center, inner sep=0, pos=0.62}, "{\tiny W_\lambda(q)}"', color={rgb,255:red,128;green,128;blue,128}, no head, from=6-1, to=6-6]
			\arrow[""{name=6, anchor=center, inner sep=0}, "{\tiny e_{\lambda'}}"', color={rgb,255:red,128;green,128;blue,128}, no head, from=6-6, to=1-6]
			\arrow[shorten <=8pt, shorten >=8pt, from=6, to=0]
			\arrow[shorten <=8pt, shorten >=8pt, from=0, to=4]
			\arrow[shorten <=5pt, shorten >=5pt, from=0, to=3]
			\arrow[shorten <=9pt, shorten >=9pt, from=2, to=3]
			\arrow["{\tiny H'}", color={rgb,255:red,214;green,92;blue,92}, shorten <=2pt, shorten >=2pt, dashed, thick, from=0, to=2]
			\arrow[shorten <=5pt, shorten >=5pt, from=5, to=0]
			\arrow["{\tiny H}", color={rgb,255:red,214;green,92;blue,92}, shorten <=2pt, shorten >=2pt, dashed, thick, from=1, to=0]
			\arrow[shorten <=12pt, shorten >=12pt, from=6, to=1]
		\end{tikzcd}
		\caption{The Macdonald square. The filled arrows are the known positivity statements; the red dashed arrows are Haglund's conjecture ($H$) and its dual ($H'$); the blue dashed arrows are \autoref{conj:ownconj} together with its dual version.}
		\label{fig:macsquare}
	\end{figure}
	
	The filled arrows are the known positivity statements, and each records a basis expanding positively in the basis at its tip. The red dashed arrows indicate Haglund's conjecture ($H$) which asserts that for $k \in \ZZ_{\geq 1},$
	\begin{align*}
		\dfrac{J_\lambda(q,q^k)}{(1-q)^{|\lambda|}} \qquad \hbox{is Schur positive,}
	\end{align*} and ($H'$) is its equivalent dual (\cite{Bhattacharya22}).
	Here $J_\lambda$ denotes the integral form Macdonald polynomials. The blue dashed arrows indicate the following conjecture along with its dual version.
	\begin{conjecture}\label{conj:ownconj}{\textup{\cite[Conjecture 5]{Bhattacharya_Thesis}}}
		The normalized integral form Macdonald polynomials expand positively in the Hall--Littlewood $P$ basis: for $k \in \ZZ_{\geq 1}$, $$\dfrac{{J_{\lambda}(t^k,t)}}{(1-t)^{|\lambda|}} \in \mathbb{Z}_{\geq 0}[t]\{ P_\mu(t) : \mu \in \Par \}.$$
	\end{conjecture}
	We have verified \autoref{conj:ownconj} in \texttt{SageMath} for $|\lambda| \leq 14$ and $1 \leq k \leq 10$. In this paper we don't make further comments on it.
	These conjectures, along with the known statements, serve to illustrate the philosophy that in the first quadrant of the $(q,t)$-plane, one obtains positivity statements as one moves left and upwards.
	
	In fact, Alexandersson \cite{Alexandersson_LLTelementarymelting} proved that $W_\lambda(q+1)$ expands positively in the elementary basis, and we prove in this paper (\autoref{thm:tplusn}) that $P_\lambda(t)$ expands positively in the $P(q)$-basis along $(q,t) = (t+n,n)$ for every $n \in \ZZ_{\geq 0}$, which, using Macdonald duality (\cite[Ch. VI, (5.1)]{MacMainBook}) implies that $W_\lambda(t+n)$ is $W_\lambda(n)$-positive. \autoref{conj:B} is another statement in this regard, which says that there is positivity as one moves upward along the positive part of $q=0$ axis, or dually, leftwards along the positive part of $t=0$ axis.
	
	\subsection*{Organization of the paper}
	In \S\ref{sec:prelims} we recall preliminaries, especially about Hall--Littlewood functions. In \S\ref{sec:basic} we prove some properties of $b_{\lambda,\mu}(q,t)$, and prove positivity for $(q,t)=(t+n,n)$ for $n \in \ZZ_{\geq 0}$ (\autoref{thm:tplusn}). In \S\ref{sec:stability} we prove a stability property and record some reduction principles, which simplifies computations in practice. In \S\ref{sec:singlerow} we compute $b_{\lambda,\mu}(q,t)$ when $\lambda$ is a single row and also give a plethystic expression for the case when $\mu$ is a single column. In \S\ref{sec:hook} we prove \autoref{thm:main}, which gives a tableaux formula for $b_{\lambda,\mu}(q,t)$ when $\lambda$ is a hook. In \S\ref{sec:twocol} we prove \autoref{conj:B} for two-column shapes (\autoref{thm:twocol-intro}). In \S\ref{sec:smallcases} we record some small cases and coefficient data, with more tableaux formulas. In \S\ref{sec:nk1} we give an expansion of $Q'_{(n,k,1)}(q)$ in the $Q'(t)$-basis (\autoref{thm:nk1}) which can also be viewed as a tableaux formula. We collect open questions in \S\ref{sec:open}.
	
	\section{Preliminaries}\label{sec:prelims}

In this section we recall known results about $q$-Pochhammer symbols and $q$-binomial coefficients, symmetric functions and Hall-Littlewood functions, including their monomial expansion formula and Kostka-Foulkes polynomials. 
	
		\subsection{$q$-binomial coefficients}\label{sec:qbin}
	
	For $n \in \ZZ_{\geq 0}$, the \textit{$q$-Pochhammer symbols} are
	\begin{align*}
		(a;q)_n = \prod_{i=0}^{n-1} (1-aq^i)
	\end{align*} with the convention that the empty product is $1$, so $(a;q)_0 = 1$. The \textit{$q$-binomial coefficients} are defined for $n \in \ZZ_{\geq 0}$ and $k \in [0,n]$ by \begin{align*}
		{\qbinom{n}{k}}_q = \dfrac{(q;q)_n}{(q;q)_{n-k}(q;q)_k}.
	\end{align*} For other cases ${{\qbinom{n}{k}}}_q =0$. In particular, the \textit{$q$-integer} $[m]_q = {m \brack 1}_q = 1+\ldots + q^{m-1}$ for $m \geq 1$. It turns out that ${\qbinom{n}{k}}_q \in \ZZ_{\geq 0}[q]$. They satisfy the following \textit{$q$-Pascal identity}:
	\begin{align}\label{eq:qPascal}
		{\qbinom{n}{k}}_q = {\qbinom{n-1}{k-1}}_q + q^{k} {\qbinom{n-1}{k}}_q.
	\end{align}
	%	Inverting $q$ in the definition gives
	%	\begin{align}\label{eq:qbinrev}
		%		{\qbinom{n}{k}}_{q^{-1}} = q^{\binom{n-k}{2}+\binom{k}{2} - \binom{n}{2}} {\qbinom{n}{k}}_q = q^{-k(n-k)}{\qbinom{n}{k}}_q.
		%	\end{align}
	%	The $q$-binomial theorem is
	%	\begin{align}\label{eq:qbin}
		%		\prod_{j=0}^{n-1} (1+q^j z) = \sum_{k \geq 0}^{} q^{\binom{k}{2}} {\qbinom{n}{k}}_q z^k,
		%	\end{align}
	%	and inverting $q$,
	%	\begin{align}\label{eq:qbininv}
		%		\prod_{j=a}^{b-1} (1+q^{-j} z) = \sum_{k \geq 0}^{} q^{\binom{k+1}{2}-bk} {\qbinom{b-a}{k}}_q z^k.
		%	\end{align}
	We will use the following formula that $q$-counts subsets:
	\begin{align}\label{eq:qcountingsubsets}
		\sum_{S \in \binom{[n]}{k}}^{} q^{\sum_{s \in S} s} = q^{\binom{k+1}{2}} {\qbinom{n}{k}}_q.
	\end{align}
	%	We also write $[m]_t = 1 + t + \cdots + t^{m-1}$ for the $t$-analogue of $m \in \ZZ_{\geq 0}$.

	\subsection{Partitions}\label{sec:partitions}
	A \textit{partition} is a sequence $\lambda = (\lambda_1,\lambda_2,\ldots)$ of non-negative integers which is weakly decreasing and has a finite sum. Let $\Par$ denote the set of all partitions. A partition $\lambda$ may be realized as a set of boxes, upper-left justified and arranged in rows so that there are $\lambda_i$ boxes in row $i$; this is called the \textit{Young diagram} of $\lambda$. For a partition $\lambda$ the \textit{conjugate partition} $\lambda'$ is the partition obtained by interchanging the rows and columns in the diagram of $\lambda$. The \textit{length} of a partition $\lambda$ is $\ell(\lambda) = \lambda'_1$ and the \textit{size} is $|\lambda|=\sum \lambda_i$. We assign coordinates to the boxes of the diagram so that $(i,j)$ denotes the box in the $i$th row from the top and the $j$th column from the left.
	
	The set of partitions is partially ordered by the \textit{dominance order} defined by \begin{align}
		\lambda \geq \mu \iff  \lambda_1+\cdots+\lambda_i \geq \mu_1+\cdots+\mu_i \ \forall \ i\geq 1, \text{ and } \ |\lambda|=|\mu| .
	\end{align}
	
	For a box $x=(r,c)$ in the diagram of $\lambda$, where $r,c$ are the row and column numbers of $x$, we write $x\in \lambda$ and define the arm length $\arm(x)$, coarm length $\coarm(x)$, leg length $\leg(x)$ and coleg length $\coleg(x)$ of $x$ by:
	\begin{align}
		\arm(x) = \lambda_r - c &, \quad \coarm(x) = c-1, \\
		\leg(x) = \lambda'_c - r &, \quad \coleg(x) = r-1.
	\end{align}
	
	\begin{figure}[h]
		\centering
		\begin{minipage}{0.6\textwidth}
			\centering
			\begin{tikzpicture}[yscale=-0.6, xscale=0.6]
				\draw (0,0) -- (8,0);
				\draw (0,1) -- (8,1);
				\draw (0,2) -- (7,2);
				\draw (0,3) -- (4,3);
				\draw (0,4) -- (4,4);
				\draw (0,5) -- (1,5);
				
				\draw (0,0) -- (0,5);
				\draw (1,0) -- (1,5);
				\draw (2,0) -- (2,4);
				\draw (3,0) -- (3,4);
				\draw (4,0) -- (4,4);
				\draw (5,0) -- (5,2);
				\draw (6,0) -- (6,2);
				\draw (7,0) -- (7,2);
				\draw (8,0) -- (8,1);
				
				\fill[gray!20] (3,1) rectangle (4,2);
				\draw[<->, thick, blue] (4,1.5) -- (7,1.5);
				\draw[<->, thick, red] (3,1.5) -- (0,1.5);
				\draw[<->, thick, green!60!black] (3.5,1) -- (3.5,0);
				\draw[<->, thick, purple] (3.5,2) -- (3.5,4);
				\draw (3.5,1.5) node {$\mathbf{x}$};
			\end{tikzpicture}
		\end{minipage}%
		\begin{minipage}{0.3\textwidth}
			\centering
			\begin{align*}
				\lambda= (8, 7, 4, 4, 1), x= (2,4)
				\\
				{\color{blue}\arm}(x)= 3, {\color{red}\coarm}(x)= 3,
				\\
				{\color{purple}\leg}(x)= 2, {\color{green!60!black}\coleg}(x)= 1.
			\end{align*}
		\end{minipage}
		\caption{Young diagram for $\lambda = (8, 7, 4, 4, 1)$ and arm, coarm, leg, coleg lengths of $x = (2,4)$}
		\label{fig:Youngdiagram}
	\end{figure}
	
	We will also frequently use, for $i \geq 1$, \begin{equation}
		m_i(\lambda) = \# \{ j \geq 1 \,| \, \lambda_j = i \},
	\end{equation} the multiplicity of $i$ in $\lambda$, and the statistic
	\begin{align}\label{eq:nlambda}
		n(\lambda) = \sum_{i \geq 1}^{} (i-1)\lambda_i,
	\end{align}
	so in particular $n(1^\ell) = \binom{\ell}{2}$.
	
	Let $\lambda,\mu$ be two partitions and $\mu \subseteq \lambda$, which means the diagram of $\mu$ is contained in the diagram of $\lambda$, or equivalently, $\mu_i \leq \lambda_i$ for all $i$. When $\mu \subseteq \lambda$ the set-theoretic difference $\lambda -\mu$ is a \textit{skew diagram}. $\theta = \lambda -\mu$ is called a \textit{horizontal strip} if it has at most one box in each column, i.e., $\theta'_i \leq 1$ for all $i$. Note that $\lambda - \mu$ is a horizontal strip if and only if $\mu$ interlaces $\lambda$: $\lambda_1 \geq \mu_1 \geq \lambda_2 \geq \mu_2 \geq \cdots $. Dually, we have the notion of a vertical strip, which means $\lambda'-\mu'$ is a horizontal strip. If we have a horizontal strip with $n$ boxes then we write $\lambda-\mu \cong (n)$ and for vertical strips we write $\lambda-\mu \cong (1^n)$.
	
	\subsection{Symmetric functions and plethysm}\label{sec:symfun}
	
	The ring of \textit{symmetric functions} in variables $x_1,x_2,\ldots$ with coefficients in a field $\mathtt{k}$ is
	\begin{align*}
		\Lambda = \{f \in \mathtt{k}[[x_1,x_2,\ldots]]\,|\, f|_{x_i \leftrightarrow x_j} = f \, \hbox{ for all } i,j \,\hbox{ and } \mathrm{deg}(f)<\infty  \},
	\end{align*} where $f|_{x_i \leftrightarrow x_j}$ denotes the operation of interchanging $x_i$ and $x_j$ in $f$.
	Since we will be dealing with different sets of parameters $q,t,\ldots,u$, we let $\mathtt{k}$ be a field containing $\QQ(q,t,\ldots,u)$. The \textit{power sums} $p_k(X) = \sum_{i\geq 1}^{} x_i^k$ are algebraically independent over $\mathtt{k}$, and $\Lambda = \mathtt{k}[p_1(X),p_2(X),\ldots]$ is a polynomial ring in infinitely many commuting indeterminates. Besides the power sums we use the \textit{monomial symmetric functions} $m_\lambda$, obtained by symmetrizing $x^\lambda$; the \textit{elementary symmetric functions} $e_k = m_{(1^k)}$; and the \textit{complete homogeneous symmetric functions} $h_k = \sum_{\lambda \vdash k}^{} m_\lambda$. For $\lambda \in \Par$ we set $p_\lambda = \prod_i p_{\lambda_i}$, and similarly $e_\lambda, h_\lambda$; the sets $\{p_\lambda\}, \{m_\lambda\}, \{e_\lambda\}, \{h_\lambda\}$ are all bases of $\Lambda$.
	
	The ring $\Lambda$ is the free $\lambda$-ring on a single generator $X = p_1$ (\cite[Ch. I, Remark 2.15]{MacMainBook}, \cite[Section 2]{Mellit_Poincarepols}), the generator being interchangeable with the power sum $p_1$. In this paper a $\lambda$-ring may be thought of as a ring with an action of the space of symmetric functions $\Lambda$ on it, and we write the action in plethystic notation $f[A]$. For us the coefficients of $\Lambda$ will most often come from the ring (or its subrings) $\QQ(q,t,\ldots,u)$, and we set $p_n[f(q,t,\ldots,u)] = f(q^n,t^n,\ldots,u^n)$. The most important examples for us will be
	\begin{align*}
		p_n\bigg[X\dfrac{1-q}{1-t}\bigg] = \dfrac{1-q^n}{1-t^n} p_n \qquad\hbox{ and }\qquad p_n\bigg[\dfrac{1-q}{1-t}\bigg] = \dfrac{1-q^n}{1-t^n},
	\end{align*} and, for a symmetric function $f$, $$ f\bigg[\dfrac{1}{1-q}\bigg] = f(1,q,q^2,\ldots).$$
	
	The $\omega$ involution on $\Lambda$ is the ring homomorphism that interchanges $h_n$ and $e_n$ for every $n$. It can also be written as $$ \omega(f) = f[-uX]|_{u=-1} \qquad \hbox{ for } f \in \Lambda.$$

	\subsection{Cauchy identity}\label{sec:cauchy}
	
	Recall that the Hall scalar product on $\Lambda$ is defined by $$ \langle h_\lambda, m_\mu \rangle = \delta_{\lambda,\mu} \qquad \hbox{ for } \lambda,\mu \in \Par.$$
	
	Let $\Exp = \sum_{n \geq 0}^{} h_n \in \widehat{\Lambda}$, where $\widehat{\Lambda}$ is the completion of $\Lambda$. Suppose $(u_\lambda)_{\lambda \in \Par}$ and $(v_\lambda)_{\lambda \in \Par}$ are bases of $\Lambda$, dual with respect to the Hall scalar product. The Cauchy identity says that, in the completion of $\Lambda \otimes \Lambda$ (the free $\lambda$-ring on two generators $X,Y$),
	\begin{align}\label{eq:Cauchy}
		\Exp[XY] = \sum_{\lambda \in \Par}^{} u_\lambda[X] v_\lambda[Y].
	\end{align}
	
	If $(u_\lambda)_{\lambda \vdash n}$ and $(v_\lambda)_{\lambda \vdash n}$ are dual bases of homogeneous degree $n$ symmetric functions then 
	\begin{equation}\label{eq:hnCauchy}
		h_n[XY] = \sum_{\lambda \vdash n}^{} u_\lambda[X] v_\lambda[Y],
	\end{equation}
	and dualizing with the $\omega$-involution, we get
	\begin{equation}\label{eq:enCauchy}
		e_n[XY] = \sum_{\lambda \vdash n}^{} \omega(u_\lambda)[X] v_\lambda[Y].
	\end{equation}
	
	In \autoref{th:b_singlerow} we show that the Cauchy identity coupled with plethystic principal specialization gives the single row coefficients $b_{(n),\mu}(q,t)$ for $\mu \vdash n$.
	
	\subsection{Hall--Littlewood functions}\label{sec:HL}
	
	The $t$-analogue of the Hall scalar product is defined by $$\langle f,g \rangle_t = \bigg\langle f, g\bigg[\dfrac{X}{1-t}\bigg]\bigg\rangle \qquad \hbox{ for } f,g \in \Lambda.$$
	
	The following theorem serves as a definition of Hall--Littlewood functions.
	\begin{theorem}
		For $\lambda \in \Par$, there is a unique symmetric function $P_\lambda(t) \in \Lambda[t]$ determined by \begin{enumerate}
			\item $P_\lambda(t)  \in s_\lambda + \sum_{\mu< \lambda}^{} \ZZ[t]s_\mu,$
			\item $\langle P_\lambda(t), P_\mu(t) \rangle_t = 0$ if $\mu \neq \lambda$.
		\end{enumerate}
	\end{theorem}
	The functions $P_\lambda(t)$ are called \textit{Hall--Littlewood $P$-functions}, or just \textit{Hall--Littlewood functions}.
	
	The \textit{Hall--Littlewood $Q$-functions} are the basis dual to the $P$-functions with respect to $\langle-,-\rangle_t$, and the \textit{Hall--Littlewood $Q'$-functions} (also called \textit{transformed Hall--Littlewood functions}) are the basis dual to the $P$-functions with respect to the usual $\langle-,-\rangle$. With these definitions,
	\begin{align}
		\langle P_\lambda(t),Q_\mu(t) \rangle_t = \langle P_\lambda(t), Q'_\mu(t) \rangle = \delta_{\lambda,\mu} \qquad \hbox{ for } \lambda,\mu \in \Par.
	\end{align}
	
	\noindent By \cite[Ch. III, (2.11),(2.12)]{MacMainBook}, \begin{equation}\label{eq:PtoQ}
		P_\lambda(t) = \dfrac{1}{\prod_{i\geq 1}(t;t)_{m_i(\lambda)}} Q_\lambda(t) \qquad \hbox{ for } \lambda \in \Par.
	\end{equation}
	
	\noindent Using the plethystic calculus,
	\begin{align}\label{eq:Q'=Qpleth}
		Q'_\lambda(t) = Q_\lambda(t)\bigg[\dfrac{X}{1-t}\bigg] \qquad \hbox{ for } \lambda \in \Par.
	\end{align}
	Note that
	\begin{align}
		Q'_\lambda(0) = s_\lambda \qquad \hbox{ and } \qquad Q'_\lambda(1) = h_\lambda \qquad \hbox{ for } \lambda \in \Par.
	\end{align}
	\noindent
	It is not hard to deduce from the triangularity and orthogonality that
	\cite[Ch. III, (2.13)]{MacMainBook}: \begin{equation}\label{eq:Qn=hnpleth}
		Q_{(n)}(t)[X] = h_n[(1-t)X] \qquad \hbox{ and } \qquad Q'_{(n)}(t) = h_n \qquad \hbox{ for } n \in \ZZ_{\geq 0}.
	\end{equation}
	By \eqref{eq:PtoQ} and \eqref{eq:Qn=hnpleth} we get
	\begin{equation}\label{eq:Pn=hpleth}
		P_{(n)}(t)[X] = \dfrac{1}{(1-t)}h_n[(1-t)X] \qquad \hbox{ for } n \in \ZZ_{\geq 1}.
	\end{equation}
	
	\subsection{The principal specialization formula for Hall--Littlewood polynomials}\label{sec:princspec}
	
	For lack of a better reference, we deduce the principal specialization formula for Hall--Littlewood functions from that of the Macdonald functions. Macdonald introduced \cite[Ch. VI]{MacMainBook} a family of symmetric functions $P_\lambda(q,t)$ depending on two parameters $q$ and $t$, which specialize to the Hall--Littlewood $P$ functions at $q=0$. By \cite[Ch. VI, (6.17)]{MacMainBook},
	\begin{align*}
		P_\lambda(q,t)\bigg[\dfrac{1-u}{1-t}\bigg] = \prod_{x \in \lambda}^{} \dfrac{t^{\coleg(x)}-q^{\coarm(x)}u}{1-q^{\arm(x)}t^{\leg(x)+1}},
	\end{align*}
	which at $q=0$, using that the boxes with $\coarm = 0$ are the ones in the first column while the boxes with $\arm = 0$ are the boxes in the right boundary, gives
	\begin{align*}
		P_\lambda(t)\bigg[\dfrac{1-u}{1-t}\bigg] = t^{n(\lambda)} \prod_{i \geq 1}^{}\dfrac{1}{(t;t)_{m_i(\lambda)}} (u;t^{-1})_{\ell(\lambda)},
	\end{align*}
	and hence
	\begin{align}\label{eq:Qlamprinc}
		Q_\lambda(t)\bigg[\dfrac{1-u}{1-t}\bigg] = t^{n(\lambda)}  (u;t^{-1})_{\ell(\lambda)}.
	\end{align}
	
	\subsection{Raising operators}\label{sec:raising}
	
	We recall the construction of $Q'_\mu(t)$ in terms of raising operators, this will be used in the proof of positivity of $b_{\lambda,\mu}(t+n,n)$ for $n \in \ZZ_{\geq 0}$ in \autoref{thm:tplusn}.
	
	Let $\widehat{\Lambda}$ denote the completion of the ring of symmetric functions, i.e, infinite linear combinations of $(h_\lambda)_{\lambda \in \Par}$. Let $n \in \ZZ_{\geq 1}$ and define $\mathtt{k}$-linear maps $$\mathbf{h}^n: \mathtt{k}[[x_1^{\pm 1},\ldots, x_n^{\pm 1}]] \to \widehat{\Lambda}$$ by $$\mathbf{h}^n\bigg(\sum_{\alpha \in \ZZ^n}^{}c_\alpha x^\alpha\bigg) = \sum_{\alpha \in \ZZ_{\geq 0}^n }^{} c_\alpha h_\alpha,$$ for coefficients $c_\alpha \in \mathtt{k}.$ 
    When $n$ is clear from the context we remove the $n$ from the notation.
	
	For a partition $\mu$ with $\ell(\mu) \leq n$, set $R^n_+ = \{\varepsilon_i - \varepsilon_j \,|\, 1 \leq i<j \leq n\} \subset \ZZ^n$, which is the set of positive roots for the root system of type $A_{n-1}$. Then \cite[(2)]{Lascoux_addingpm1},\cite[Ch. III, ($2.15^\prime$)]{MacMainBook} says
	\begin{align}\label{eq:Qp=hraisingop}
		Q'_\mu(t) = \mathbf{h}\bigg( \prod_{\alpha \in R^n_+}^{} \dfrac{1-x^\alpha}{1-tx^\alpha} x^\mu \bigg)
	\end{align} where $\dfrac{1}{1-tx^\alpha} = \sum_{n \in \ZZ_{\geq 0}}^{} t^n x^{n \alpha}$ as an element of the formal power series ring. Since the elements of $R^n_+$ are of degree $0$, the right hand side of \eqref{eq:Qp=hraisingop} lands in homogeneous symmetric functions of degree $|\mu|$, in particular, it is an element of $\Lambda.$
	
	Note that the above can be used to define $Q'_\mu(t)$ for any $\mu \in \ZZ^n$. By \cite[(3)]{Lascoux_addingpm1}, \cite[Ch. III.2, Ex. 2]{MacMainBook} they satisfy the straightening rules
	\begin{align*}
		Q'_{(\ldots,s,r,\ldots)}(t) = tQ'_{(\ldots,r,s,\ldots)}(t) - Q'_{(\ldots,r-1,s+1,\ldots)}(t) + t Q'_{(\ldots,s+1,r-1,\ldots)}(t).
	\end{align*}
	These can be iterated to give
	\begin{proposition}{\textup{\cite[Ch. III.2, Ex. 2]{MacMainBook}}}\label{prop:Qpstraightening}
		Let $s<r$. 
		If $r-s = 2m+1$, then
		\begin{align}
			Q'_{(s,r)}(t) &= t Q'_{(r,s)}(t) + \sum_{i=1}^{m}(t^{i+1}-t^{i-1}) Q'_{(r-i,s+i)}(t);
			\intertext{if $r-s = 2m$, then}
			Q'_{(s,r)}(t) &= t Q'_{(r,s)}(t) + \sum_{i=1}^{m-1}(t^{i+1}-t^{i-1}) Q'_{(r-i,s+i)}(t)  + (t^m-t^{m-1})Q'_{(r-m,s+m)}(t);
		\end{align}
		and the same holds for any two consecutive indices of $Q'$.
	\end{proposition}
	
	\subsection{Monomial expansion}\label{sec:monexpform}
	
	Recall that a \textit{semistandard Young tableau} $T$ of shape $\lambda \in \Par$ is a filling of the Young diagram of $\lambda$ with numbers in $\ZZ_{>0}$ such that the entries in each row weakly increase from left to right and the entries in each column strictly increase from top to bottom. The content of such a tableau is the sequence whose $i$th term is the number of $i$'s in the tableau. The set of semistandard Young tableaux of shape $\lambda$ and content $\mu$ is denoted by $\SSYT(\lambda,\mu)$.
	
	Next, we recall Macdonald's monomial expansion formula for $P_\lambda(t)$. We remark that how we write the formula is the box-by-box version of the one found in \cite[Ch. III, (5.9$^\prime$)]{MacMainBook}, and the connection may be found in \cite[\S 4]{Bhattacharya2024monomialexpansionformulahalllittlewood}.
	
	For a box $x \in \lambda$, let $\Leg(x,\lambda)$ denote the set of all boxes below $x$ in the same column. Let $T$ be a semistandard Young tableau of shape $\lambda$. Let $\leg_T(x,<i) = \# \{ x' \in \Leg(x,\lambda)\,|\, T(x')<i \}$.
	
	For example, in \autoref{fig:egpsiTx<}, the gray boxes denote the set $\Leg(x,\lambda)$ for the bulleted box $x=(2,2)$ and $\lambda=(4,3,3,2,1)$. The diagram on the right shows a semistandard Young tableau $T \in \SSYT(\lambda,\mu)$ with $\mu = (3,3,2,2,2,1)$. There
	$\leg_T(x,<i)=0$ for $i\in\{1,2,3,4\}$ and
	$\leg_T(x,<5)=1,\ \leg_T(x,<6)=2$.
	\begin{figure}[htbp]
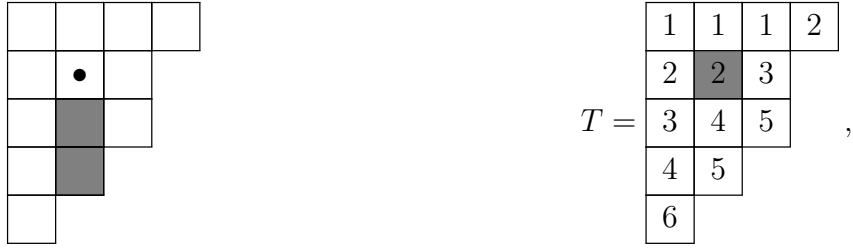

		\centering
		\begin{minipage}{0.48\textwidth}
			\centering
			\ytableausetup{nosmalltableaux,centertableaux}
			\[
			\ydiagram{4,3,3,2,1}*[*(gray)]{0,0,1+1,1+1}*[*(white) \bullet ]{0,1+1}
			\]
		\end{minipage}
		\hfill
		\begin{minipage}{0.48\textwidth}
			\centering
			\[
			T = \ytableaushort{1112,2{*(gray)2}3,345,45,6}\,,
			\]
		\end{minipage}
		\caption{Example of $\leg_T(x,<i)$}
		\label{fig:egpsiTx<}
	\end{figure}
	
	For a box $x = (r,c) \in \lambda$ define \begin{equation}\label{eq:psiTx}
		\psi_T(x;t) = \begin{cases}
			1- t^{\leg_T(x,<i)+1}\,, & \hbox{ if } T(r,c+1) = i \hbox{ and } i \notin T(\Leg(x,\lambda))\cup \{T(x)\},
			\\
			\qquad 1\,, & \hbox{ otherwise}.
		\end{cases}
	\end{equation} When $c=\lambda_r$, $\psi_T(x;t) = 1$.
	Define
	\begin{equation}
		\psi_T(t) = \prod_{x \in \lambda} \psi_T(x;t).
	\end{equation}
	Then
	\begin{theorem}{\textup{\cite[Ch. III, (5.11$^\prime$)]{MacMainBook}}}
		For $\lambda \in \Par$,
		\begin{equation}
			P_\lambda(t) = \sum_{\mu \leq \lambda}^{} \bigg( \sum_{T \in \SSYT(\lambda,\mu)}^{} \psi_T(t) \bigg) m_\mu.
		\end{equation}
	\end{theorem}
	
	In particular, if $\lambda$ is of hook shape, i.e., $\lambda = (n,1^k)$ for some $n \in \ZZ_{> 1}$ and $k \in \ZZ_{\geq 0}$, then for any $T$ of shape $\lambda$ whose content is a partition, $$\psi_T(t) = [m]_t(1-t)^{r-1}$$ where $r$ is the number of distinct entries in the first row of $T$, and $$m = \begin{cases}
		T(1,2)-1, &\hbox{ if } T(1,2) \neq 1,
		\\
		\qquad 1, &\hbox{ otherwise.}
	\end{cases}$$
	
	For example, \begin{align*}
		\psi_{\ytableaushort{12,3}}(t) =  (1-t), \qquad \hbox{ and } \qquad \psi_{\ytableaushort{13,2}}(t) = (1-t^2),
	\end{align*}
	giving
	\begin{align*}
		P_{(2,1)}(t) = m_{(2,1)} + \bigg((1-t)+(1-t^2)\bigg) m_{(1^3)}.
	\end{align*}
	
	We generalize this formula for $b_{\lambda,\mu}(q,t)$ when $\lambda$ is a hook shape in \S\ref{sec:hook}.

%\subsection{Pieri rules}
%Closely related to the monomial expansion formula is the vertical strip Pieri rule for Hall-Littlewood polynomials, which we now discuss. This will be crucial for the proof of our theorems.
%
%From the Pieri rule for Macdonald polynomials \cite[Ch. VI, (6.24)]{MacMainBook} for the product $P_\nu(q,t) e_r$, specialized to $q=0$, we have for $\nu \in \Par$ and $m \in \ZZ_{\geq 0}$,
%\begin{equation}\label{eq:Pieri}
%	P_\nu(t)e_m =   \sum_{\lambda - \nu \cong (1^m)}^{} \psi'_{\lambda/\nu}(0,t) P_\lambda(t),
%\end{equation}
%where, $\lambda - \nu$ is a vertical strip of size $m$ and
%\begin{equation}\label{eq:psidef}
%	\psi'_{\lambda/\nu}(0,t) = \prod_{i}^{} {\qbinom{\lambda'_i - \lambda'_{i+1}}{\lambda'_i - \nu'_i}}_t.
%\end{equation}
%In particular, for $n-k \geq 1$ and $k \geq 1$,
%\begin{align}\label{eq:PnemPieri}
%	P_{(n-k)}(t)e_k = P_{(n-k,1^k)}(t) + P_{(n-k+1,1^{k-1})}(t).
%\end{align}
	
	\subsection{Kostka--Foulkes polynomials}\label{sec:kostkafoulkes}
	
	Write $$ s_\lambda = \sum_{\mu \in \Par}^{} K_{\lambda,\mu}(t) P_\mu(t).$$
	Note that $K_{\lambda,\lambda}(t) = 1$ and $K_{\lambda,\mu}(t) = 0$ unless $\mu \leq \lambda$.
	
	\begin{theorem}{\textup{\cite{LS_suruneFoulkes}}}
		For $\lambda,\mu \in \Par$,	$K_{\lambda,\mu}(t) \in \ZZ_{\geq 0}[t]$.
	\end{theorem}
	
	Lascoux and Sch\"utzenberger found a statistic \textit{charge} $\ch: \SSYT \to \ZZ_{\geq 0}$ and proved that $$K_{\lambda,\mu}(t) = \sum_{T \in \SSYT(\lambda,\mu)}^{} t^{\ch(T)}.$$
	
	\subsection{Charge}\label{sec:charge}
	
	We give a brief description of charge; some better references are \cite[Ch. III.6]{MacMainBook}, \cite[Ch. 5]{Lothaire_AConwords}.
	
	For $\lambda,\mu \in \Par$ and $T \in \SSYT(\lambda,\mu)$, the charge $\ch(T) \in \ZZ_{\geq 0}$ is defined as follows.
	
	First, form the \textit{reverse reading word} $w$ by reading the rows of $T$ from right to left, starting with the top row and proceeding to the bottom.
	
	Next, from this word, we extract a sequence of standard subwords $w^1, w^2, \ldots, w^{\mu_1}$. To construct $w^1$: begin by selecting the first occurrence of $1$ in the word (from the left). After selecting $i$, choose the first occurrence of $i+1$ that appears to its right; if none exists, then circle back and start from the beginning of the word. Continue this process until the maximum letter is reached. Once $w^1$ is formed, remove its letters from the word and repeat on the remaining sequence to obtain $w^2, w^3, \ldots$ in turn.
	
	The charge of $w$ is then the sum of the charges of its standard subwords. For a standard subword, assign to each letter $i$ an index: set $I(1)= 0$, and for $i \geq 1$, \begin{align*}
		I(i+1) = \begin{cases}
			I(i), & \text{if } i+1 \text{ appears to the right of } i,\\
			I(i) + 1, & \text{if } i+1 \text{ appears to the left of } i.
		\end{cases}
	\end{align*}
	The charge of the subword is $\sum_{i}^{} I(i)$, the sum of the indices of all of its letters.
	
	For example, for the tableau $T$ from \autoref{fig:egpsiTx<}, $w = 2111322543546$. The standard subwords are $w^1 = 125346, w^2 = 13245, w^3 = 21$ with charges $2, 3, 1$ respectively, giving $\ch(T) = 2+3+1 = 6$.
	
	Note that this construction assigns to each box $x$ of the Young diagram of $T$ a weight $\ch_T(x)$ (namely, the index of the letter occupying $x$ within its standard subword), and the total charge $\ch(T)$ is the sum of these weights. Our weights $\psi_T(q,t)$ in \S\ref{sec:hook} and \S\ref{sec:smallcases}, \S\ref{sec:nk1} are based on this box-by-box version of charge.

	\section{Basic properties of $b_{\lambda,\mu}(q,t)$}\label{sec:basic}
	
	In this section we recall our main object of study, the $b_{\lambda,\mu}(q,t)$'s. We prove that $b_{\lambda,\mu}(t+n,n) \in \ZZ_{\geq 0}[t]$ by using raising operator techniques.
	
	\subsection{Definition and triangularity}
	We recall the definition of $b_{\lambda,\mu}(q,t)$ from \S\ref{sec:intro}.
	
	\begin{definition}\label{Bdef}
		For $\lambda,\mu \in \Par$, define $b_{\lambda,\mu}(q,t) \in \ZZ[q,t]$ by one of the following equivalent equations: \begin{enumerate}
			\item $P_{\lambda}(t) = \sum_{\mu}b_{\lambda,\mu}(q,t)P_{\mu}(q)$,
			\item $Q'_{\mu}(q) = \sum_{\lambda}b_{\lambda,\mu}(q,t) Q'_{\lambda}(t)$,
			\item $b_{\lambda,\mu}(q,t) = \langle P_{\lambda}(t), Q'_{\mu}(q) \rangle$\,,
		\end{enumerate}
		where $\langle -,-\rangle$ denotes the Hall scalar product.
	\end{definition}
	
	\begin{proof}[Proof of equivalence]
		The equivalence of all three is a consequence of $(P_\lambda(q))_{\lambda \in \Par}$ and $(Q'_\lambda(q))_{\lambda \in \Par}$ being dual bases with respect to the Hall scalar product.
	\end{proof}
	
	\begin{proposition}[Triangularity]\label{prop:triangularity}
		Let $\lambda,\mu \in \Par$. Then
		$b_{\lambda,\mu}(q,t) = 0$ unless $\mu \leq \lambda$ in dominance order, and $b_{\lambda,\lambda}(q,t) = 1$.
	\end{proposition}
	\begin{proof}
		The Kostka--Foulkes matrix $K(u)$ is upper-unitriangular with respect to dominance order (\S\ref{sec:kostkafoulkes}), hence so is its inverse. Since $b(q,t) = K(t)^{-1}K(q)$ is a product of two unitriangular matrices, it is upper-unitriangular as well.
	\end{proof}
	
	We also note that  $b_{\lambda,\mu}(t,t) = \delta_{\lambda,\mu}$, since $b(t,t) = K(t)^{-1}K(t) = \mathrm{Id}$.
	
	In this paper we would like to advertise the following conjecture, already stated in the introduction (\autoref{conj:B}).
	\begin{conjecture}\label{conj:B-body}
		For $\lambda,\mu \in \Par$,	$b_{\lambda,\mu}(t+1,t) \in \mathbb{Z}_{\geq 0}[t]$.
	\end{conjecture}
	We have verified this conjecture in \texttt{SageMath} for all partitions of size $\leq 26$. We prove it when $\lambda$ is a hook (\autoref{thm:main}) and when $\lambda_1 \leq 2$ (\autoref{thm:twocol}), when $\mu = (n,k,1)$ for $n \geq k+1 \geq 3$ (\autoref{thm:nk1}), and some more simple cases in \S\ref{sec:smallcases}.
	
	\subsection{Proof of positivity for $(q,t)=(t+n,n)$}\label{sec:postplusn}
	
	We use standard notation for root systems of type $A$. Let $\varepsilon_i$ denote the vector in $\ZZ^m$ with all coordinates $0$ except the $i$th coordinate, which is $1$. The positive roots are $R^m_+ = \{\varepsilon_i - \varepsilon_j \,| \, 1 \leq i<j \leq m\} \subset \ZZ^m$; the positive part of the root lattice $Q^m_+$ is the set of all $\ZZ_{\geq 0}$-linear combinations of elements of $R^m_+$. For a vector $\mu \in Q^m_+$, a Kostant partition of $\mu$ is a tuple $\kappa = (c_\alpha)_{\alpha \in R^m_+}$ with $c_\alpha \in \ZZ_{\geq 0}$ such that $\sum_{\alpha \in R^m_+}^{} c_\alpha \alpha = \mu$. We let $|\kappa| = \sum_{\alpha \in R^m_+}^{} c_\alpha$, the sum of multiplicities of $\kappa$, and $\# \kappa = \# \{\alpha \in R^m_+\,| \, c_\alpha \neq 0\}$, the cardinality of the support of $\kappa$. For $\mu \in Q^m_+$, we let
	\begin{align*}
		\cP(\mu;q,t) = \sum_{\kappa}^{} (q-t)^{\# \kappa} q^{|\kappa|-\# \kappa},
	\end{align*} where the sum runs over all Kostant partitions $\kappa$ of $\mu$.
	Then
	\begin{align}\label{eq:prod=sumKp}
		\prod_{\alpha \in R^m_+}^{} \dfrac{1-tx^\alpha}{1-qx^\alpha} = \prod_{\alpha \in R^m_+}^{} \bigg( 1 + (q-t) \sum_{n \in \ZZ_{\geq 1}}^{} q^{n-1} x^{n \alpha} \bigg) = \sum_{\mu \in Q^m_+}^{} \cP(\mu;q,t) x^\mu
	\end{align}
	
	\begin{theorem}\label{thm:bqtpn}
		$b_{\lambda,\mu}(t+n,n) \in \mathbb{Z}_{\geq 0}[t]$ for every $n \in \mathbb{Z}_{\geq 0}$.
	\end{theorem}
	
	\begin{proof}
		For $\mu \in \Par$, taking $m \geq \ell(\mu)$, we can consider $\mu = (\mu_1,\ldots,\mu_m) \in \ZZ^m$. Using \eqref{eq:Qp=hraisingop} and \eqref{eq:prod=sumKp}, for $m$ variables, \begin{align}
			Q'_\mu(q) &= \mathbf{h}\bigg( \prod_{\alpha \in R^m_+}^{} \dfrac{1-x^\alpha}{1-qx^\alpha} x^\mu\bigg)
			\nonumber
			\\
			&= \mathbf{h}\bigg( \bigg(\sum_{\nu \in Q^m_+}^{} \cP(\nu;q,t) x^\nu\bigg) \prod_{\alpha \in R^m_+}^{} \dfrac{1-x^\alpha}{1-tx^\alpha} x^\mu \bigg)
			\nonumber
			\\
			&= \sum_{\nu \in Q^m_+}^{} \cP(\nu;q,t) Q'_{\mu+\nu}(t).
			\label{eq:Q'KostantPartitionexp}
		\end{align}
		
		\noindent By \autoref{prop:Qpstraightening}, if $t = n \in \ZZ_{\geq 1}$, then for any $\nu \in \ZZ^m$, $Q'_\nu(n)$ is a non-negative integer combination of the $Q'_\lambda(n)$ for $\lambda \in \Par$ (see \cite[Ch. III.2, Ex. 2]{MacMainBook}). Then for $q = t+n$, since $\cP(\nu;t+n,n) \in \ZZ_{\geq 0}[t]$, we get that $b_{\lambda,\mu}(t+n,n) \in \ZZ_{\geq 0}[t]$.
		
		Already, for $n=0$, $b_{\lambda,\mu}(t,0) = K_{\lambda,\mu}(t) \in \ZZ_{\geq 0}[t]$. This proves the statement for all $n \in \ZZ_{\geq 0}$.
	\end{proof}
	
	\section{Stability and reduction principles}\label{sec:stability}
	
	Computation of $b_{\lambda,\mu}(q,t)$ can often be reduced to smaller shapes in particular cases. For example if $\lambda,\mu$ both have the same first part, then one can consider the pair with the first part removed and compute it. In this section we record some more of such properties.
	
	\subsection{A constant term expression}
	
	We recall the main result from \cite{Zabrocki_thesis}, to get an expression for $b_{\lambda,\mu}(q,t)$. Let 
	\begin{equation}
		\Delta_t[Z_k] = \prod_{1 \leq i < j \leq k}^{} \dfrac{1-z_jz_i^{-1}}{1-tz_jz_i^{-1}} \qquad\hbox{ and }\qquad Z_k = z_1+\cdots + z_k.
	\end{equation}
	Then for $k \geq \ell(\mu)$,
	\begin{align}
		Q'_\mu(t)[X] = [z^\mu]\Exp[XZ_k] \Delta_t[Z_k],
	\end{align} where the right hand side is the coefficient of $z^\mu$ in $\Exp[XZ_k] \Delta_t[Z_k]$. Here we are identifying $X = x_1+x_2+\ldots$ as before.
	Then, using $ \langle \Exp[AX], f[X] \rangle = f[A] $ (see \cite[(2.6)]{HaimanGeometry}) and (3) of \autoref{Bdef} we get
	\begin{align}\label{eq:bconstterm}
		b_{\lambda,\mu}(q,t) &= [z^\mu] \Delta_q[Z_k] P_\lambda(t)[Z_k].
	\end{align}
	
	\subsection{Stability}
	
	Note that the Kostka--Foulkes polynomials satisfy the following stability property.
	
	\begin{proposition}\label{prop:Kstability}
		Suppose $\lambda \geq \mu$ are two partitions. If $\mu_1 \geq \lambda_2$ then $K_{\lambda,\mu}(t) = K_{\lambda+(r),\mu+(r)}(t)$  where $\lambda+(r)=(\lambda_1+r,\lambda_2,\ldots)$ and $r\geq 0$.
	\end{proposition}
	\begin{proof}
		There is a bijection $\SSYT(\lambda,\mu) \to \SSYT(\lambda+(r),\mu+(r))$ where a tableau $T$ maps to the tableau obtained by moving the top row $r$ squares to the right and inserting $r$ $1$'s in the squares vacated. This preserves charge, because the new standard subwords are just singleton $1$, and the old subwords are unchanged.
	\end{proof}
	
	We now prove a generalization of this for the $b_{\lambda,\mu}(q,t)$'s.
	\begin{proposition}\label{prop:stability}
		Suppose $\lambda \geq \mu$ are two partitions. If $\mu_1 > \lambda_2$ then $b_{\lambda,\mu}(q,t) = b_{\lambda+(r),\mu+(r)}(q,t)$.
	\end{proposition}
	
	\begin{proof}
		It suffices to treat the case $r=1$; the general case follows by induction. From \eqref{eq:bconstterm}, writing $\Delta_q[Z_k] = \sum_{\kappa}^{} \mathcal{P'}(\kappa;q) z^\kappa$, where $\mathcal{P'}(\kappa;q)$ denotes the coefficient of $z^\kappa$ in the expansion of $\Delta_q[Z_k]$, and extracting the coefficient of $z^{\mu+(1)}$ gives
		\begin{align*}
			b_{\lambda+(1),\mu+(1)}(q,t) &= \sum_{\kappa \in \ZZ^k}^{} \mathcal{P'}(\kappa;q) \cdot [z^{\mu+(1)-\kappa}] P_{\lambda+(1)}(t)[Z_k].
		\end{align*}
		Note that $\mu-\kappa$ may not be a partition. However, since in $\Delta_q[Z_k]$ only negative powers of $z_1$ appears, the sum only runs over $\kappa$ with $\kappa_1\leq 0$, and so $(\mu-\kappa)_1 = \mu_1 - \kappa_1> \lambda_2$. So it suffices to prove $[z^{\alpha+(1)}]P_{\lambda+(1)}(t)[Z_k] = [z^\alpha] P_\lambda(t)[Z_k]$ for a composition $\alpha$ when $\alpha_1>\lambda_2$. But this follows from the $\psi_T$ formula \eqref{eq:psiTx}, since the extra box containing $1$ has no leg.
	\end{proof}
	
	\begin{remark}\label{rem:stabilityfinitecheck}
		\autoref{prop:stability} says that whenever $\mu_1 > \lambda_2$, the coefficient $b_{\lambda,\mu}(q,t)$ depends only on the reduced pair $(\lambda-(r), \mu-(r))$ where $r=\mu_1-\max(\lambda_2+1,\mu_2)$. In particular, any infinite family of identities (or positivities) with fixed reduced pair and $\mu_1 > \lambda_2$ reduces to a single finite computation. We use this repeatedly in \S\ref{sec:smallcases}.
	\end{remark}
	
	Since \[ Q'_{\lambda}(q)|_{h_\mu} = b_{\mu,\lambda}(q,1) ,\] the above implies
	
	\begin{corollary}
		Let $\lambda,\mu \in \Par$ such that $\lambda_1 > \mu_2$. Then \begin{align*}
			Q'_\lambda(q)|_{h_\mu} = Q'_{\lambda + (1)}(q)|_{h_{\mu + (1)}}.
		\end{align*}
	\end{corollary}
	
	The strict inequality is necessary: the statement fails when $\lambda_1 = \mu_2$. For example, if $\lambda = (2,2,1), \mu = (3,2)$ then $Q'_{(2,2,1)}(q)|_{h_{(3,2)}} = q^2-1$ and $Q'_{(3,2,1)}(q)|_{h_{(4,2)}} = q^2-q$.
	
	\subsection{Reduction principles}
	
	The statements analogous to \cite{Bhattacharya22} hold true, since each of these statements is true for the Kostka--Foulkes and inverse Kostka--Foulkes matrices.
	
	\begin{proposition}\label{prop:reductionprinc}
		Let $\lambda, \mu \in \Par$ with $\lambda \geq \mu$. Then
		\begin{enumerate}
			\item If $\lambda_1 + \cdots + \lambda_k = \mu_1 + \cdots + \mu_k$ then
			\begin{align*}
				b_{\lambda,\mu}(q,t) = b_{(\lambda_1,\ldots,\lambda_k),(\mu_1,\ldots,\mu_k)}(q,t) \cdot b_{(\lambda_{k+1},\lambda_{k+2},\ldots),(\mu_{k+1},\mu_{k+2},\ldots)}(q,t).
			\end{align*}
			
			\item If $\lambda'_1 + \cdots + \lambda'_k = \mu'_1 + \cdots + \mu'_k$ then
			\begin{align*}
				b_{\lambda,\mu}(q,t) = b_{(\lambda'_1,\ldots,\lambda'_k)',(\mu'_1,\ldots,\mu'_k)'}(q,t) \cdot b_{(\lambda'_{k+1},\lambda'_{k+2},\ldots)',(\mu'_{k+1},\mu'_{k+2},\ldots)'}(q,t).
			\end{align*}
			
			\item If $\lambda,\mu \subseteq (m^k)$ and $\lambda^c, \mu^c$ are the complementary partitions to $\lambda,\mu$ inside the rectangle $(m^k)$, then
			\begin{align*}
				b_{\lambda,\mu}(q,t) = b_{\lambda^c,\mu^c}(q,t).
			\end{align*}
		\end{enumerate}
	\end{proposition}
	
	\section{Single row and single column cases}\label{sec:singlerow}
	
	In this short section we find a formula for $b_{\lambda,\mu}(q,t)$ when $\lambda$ is a single row, by combining the Cauchy identity and the principal specialization formula. We also show that the same technique produces an expression for the case when $\mu$ is a single column.
	
	\subsection{$\lambda$ a single row}
	\begin{theorem}\label{th:b_singlerow}
		Let $n \in \ZZ_{\geq 1}$ and $\mu \vdash n$. Let $\ell = \ell(\mu)$. Then \begin{equation}
			b_{(n),\mu}(q,t) = q^{n(\mu)-\binom{\ell}{2}} \prod_{i=1}^{\ell-1}(q^i-t)
		\end{equation}
	\end{theorem}
	
	\begin{proof}
		By \eqref{eq:Pn=hpleth}, $P_{(n)}(t) = \dfrac{1}{1-t} h_n[X(1-t)]$ and by the Cauchy identity \eqref{eq:hnCauchy} gives
		\begin{align*}
			h_n[X(1-t)] = \sum_{\lambda \vdash n}^{} Q'_\lambda(q)[1-t]P_\lambda(q)[X].
		\end{align*}
		
		Hence
		\begin{align*}
			b_{(n),\lambda}(q,t) = \dfrac{1}{1-t} Q'_\lambda(q)[1-t] = \dfrac{1}{1-t}Q_\lambda(q)\bigg[\dfrac{1-t}{1-q}\bigg]
			= \dfrac{1}{1-t} q^{n(\lambda)} (t;q^{-1})_{\ell(\lambda)},
		\end{align*}
		by \eqref{eq:Q'=Qpleth} and \eqref{eq:Qlamprinc}. This gives the stated form of $b_{(n),\lambda}(q,t)$.
	\end{proof}
	
	\begin{corollary}
		Let $\lambda,\mu \in \Par$ such that $K_{\lambda,\mu}=1$. Then $b_{\lambda,\mu}(t+1,t) \in \mathbb{Z}_{\geq 0}[t]$.
	\end{corollary}
	
	\begin{proof}
		This is analogous to the argument from \cite{Bhattacharya22}. The classification of pairs $(\lambda,\mu)$ with $K_{\lambda,\mu}=1$, due to Berenshte\u{\i}n and Zelevinski\u{\i} \cite{BZ90}, says that after using the reduction principles from \autoref{prop:reductionprinc}, it suffices to consider the case when $\lambda$ is a single row, which is proved above.
	\end{proof}
	
	\subsection{$\mu$ a single column}
	We can use the same trick to give a plethystic expression for $b_{\lambda,(1^n)}(q,t)$, although we have at present not found a more explicit/combinatorial expression.
	We have \cite[Remark 2.6]{Mellit_Poincarepols}
	\begin{align}
		Q'_{(1^n)}(q) = q^{\binom{n}{2}} \dfrac{e_n\bigg[\dfrac{X}{1-q}\bigg]}{e_n\bigg[\dfrac{1}{1-q}\bigg]}
		= (q;q)_n\, e_n\bigg[\dfrac{X}{1-q}\bigg].
	\end{align}
	By the dual Cauchy formula \eqref{eq:enCauchy}, $$e_n\bigg[\dfrac{X}{1-q}\bigg] = \sum_{\lambda \vdash n}^{} (\omega P_\lambda(t))\bigg[\dfrac{1}{1-q}\bigg] Q'_\lambda(t)[X],$$
	so for $\lambda \vdash n$,
	\begin{equation}\label{eq:blam1nprinc}
		b_{\lambda,(1^n)}(q,t) = (q;q)_n\, (\omega P_\lambda(t))\bigg[\dfrac{1}{1-q}\bigg].
	\end{equation}
	
	\section{A tableaux formula for hook shapes}\label{sec:hook}
	
	In this section we prove \autoref{thm:main}. Throughout, $\lambda = (n-k,1^k)$ with $n-k>1$ and $k \geq 0$. For $T \in \SSYT(\lambda,\mu)$ we recall the weight $\psi_{T}(q,t) \in \ZZ[q,t]$ from the introduction:
	\begin{align}
		\psi_{T}(x;q,t) = \begin{cases}
			(q^{\ch_T(x)}-t), &\hbox{ if } \coarm(x) > 0, \hbox{ and } T(\Left(x)) < T(x),
			\\
			\qquad q^{\ch_T(x)}, &\hbox{ otherwise}.
		\end{cases}
	\end{align}
	Let \begin{equation}
m = \begin{cases}
	T(1,2)-1, &\hbox{ if } T(1,2) \neq 1,
	\\
	\qquad 1, &\hbox{ otherwise.}
\end{cases}	\end{equation} 
	Define
	\begin{align}\label{eq:psiqtdef}
		\psi_{T}(q,t) = [m]_t \cdot \prod_{x \in \lambda}^{} \psi_{T}(x;q,t).
	\end{align}
	Note that $$ \psi_T(q,0) = q^{\ch(T)} \qquad \hbox{ and } \psi_T(1,t) = \psi_T(t).$$
	\noindent
	We now restate \autoref{thm:main}.
	\begin{theorem}\label{thm:mainlocal}
		Let $\lambda = (n-k,1^k)$ with $n-k>1$ and $k \geq 0$. Then for all $\mu \in \Par$,
		\begin{equation}
			b_{\lambda,\mu}(q,t) = \sum_{T \in \SSYT(\lambda,\mu)}^{} \psi_{T}(q,t).
		\end{equation}
	\end{theorem}
	
	\begin{remark}
		We note that even though most of the notation used in this paper is consistent with Macdonald's book \cite{MacMainBook}, our $\psi_T(q,t)$ is different from the one in their Chapter VI.
	\end{remark}

	\begin{example}\label{eg:hook31}
		Let $\lambda = (3,1)$. The partitions dominated by $\lambda$ are $(3,1)$, $(2,2)$, $(2,1,1)$ and $(1,1,1,1)$. For each $\mu$ we list the tableaux $T \in \SSYT(\lambda,\mu)$ together with their weights $\psi_T(q,t)$ from \autoref{thm:mainlocal}, and the resulting coefficient $b_{\lambda,\mu}(q,t)$.
		
		For $\mu = (3,1)$ there is a single tableau:
		\ytableausetup{smalltableaux}
		\begin{align*}
			\psi_{\ytableaushort{111,2}}(q,t) = 1,
			\qquad
			b_{(3,1),(3,1)}(q,t) = 1.
		\end{align*}
		
		For $\mu = (2,2)$:
		\begin{align*}
			\psi_{\ytableaushort{112,2}}(q,t) = q-t,
			\qquad
			b_{(3,1),(2,2)}(q,t) = q-t.
		\end{align*}
		
		For $\mu = (2,1,1)$:
		\begin{align*}
			\psi_{\ytableaushort{112,3}}(q,t) = (q-t)q,
			\qquad
			\psi_{\ytableaushort{113,2}}(q,t) = q-t,
		\end{align*}
		\begin{align*}
			b_{(3,1),(2,1,1)}(q,t) = (q-t)q + (q-t).
		\end{align*}
		
		For $\mu = (1,1,1,1)$:
		\begin{align*}
			\psi_{\ytableaushort{123,4}}(q,t) &= (q^2-t)(q-t)q^2, \\
			\psi_{\ytableaushort{124,3}}(q,t) &= (q^2-t)(q-t)q, \\
			\psi_{\ytableaushort{134,2}}(q,t) &= (q^2-t)(q-t)(t+1),
		\end{align*}
		\begin{align*}
			b_{(3,1),(1,1,1,1)}(q,t) &= (q-t)(q^2-t)(1+q+q^2+t),
		\end{align*}
		where the last factorization is the formula of \autoref{prop:Cnk}.
	\end{example}
	
	We prove \autoref{thm:mainlocal} in two cases, depending on whether $\mu$ is a single column or not; in both cases the argument rests on the Pieri recursion for hooks  which we derive first.
	
		\subsection{The Pieri recursion}
	
	The proof of \autoref{thm:mainlocal} has two steps. First, we derive a recursion that determines all hook coefficients $b_{(n-k,1^k),\lambda}(q,t)$ from the single-row coefficients, which are known explicitly by \autoref{th:b_singlerow}; this is \eqref{eq:brecPieri} below. Second, and this is the main work, we show that the tableau sum satisfies the same recursion. 
	
	We first recall the vertical strip Pieri rule for Hall-Littlewood polynomials. From the Pieri rule for Macdonald polynomials \cite[Ch. VI, (6.24)]{MacMainBook} for the product $P_\nu(q,t) e_r$, specialized to $q=0$, we have for $\nu \in \Par$ and $m \in \ZZ_{\geq 0}$,
	\begin{equation}\label{eq:Pieri}
		P_\nu(t)e_m =   \sum_{\lambda - \nu \cong (1^m)}^{} \psi'_{\lambda/\nu}(0,t) P_\lambda(t),
	\end{equation}
	where, $\lambda - \nu$ is a vertical strip of size $m$ and
	\begin{equation}\label{eq:psidef}
		\psi'_{\lambda/\nu}(0,t) = \prod_{i}^{} {\qbinom{\lambda'_i - \lambda'_{i+1}}{\lambda'_i - \nu'_i}}_t.
	\end{equation}
	In particular, for $n-k > 1$ and $k \geq 1$,
	\begin{align}\label{eq:PnemPieri}
		P_{(n-k)}(t)e_k = P_{(n-k,1^k)}(t) + P_{(n-k+1,1^{k-1})}(t).
	\end{align}
	On the other hand, expanding $P_{(n-k)}(t)$ first into the $P(q)$-basis and then applying the Pieri rule gives for $n-k>1$,
	\begin{align}\label{eq:PnemtoPq}
		P_{(n-k)}(t)e_k = \sum_{\mu \vdash n-k}^{} b_{(n-k),\mu}(q,t) P_\mu(q)e_k = \sum_{\mu \vdash n-k}^{} b_{(n-k),\mu}(q,t)\sum_{\lambda - \mu \cong (1^k)}^{} \psi'_{\lambda/\mu}(0,q) P_\lambda(q).
	\end{align}
	
	Comparing coefficients of $P_\lambda(q)$ in the two expressions, namely \eqref{eq:PnemPieri} expanded in the $P(q)$-basis and \eqref{eq:PnemtoPq}, we obtain for every $\lambda \vdash n$ the recursion
	\begin{align}\label{eq:brecPieri}
		b_{(n-k+1,1^{k-1}),\lambda}(q,t) + b_{(n-k,1^k),\lambda}(q,t) = \sum_{\substack{\mu \vdash n-k \\ \lambda - \mu \cong (1^k)}}^{} b_{(n-k),\mu}(q,t)  \prod_{i \geq 1}^{} {\qbinom{\lambda'_i - \lambda'_{i+1}}{\lambda'_i - \mu'_i}}_q.
	\end{align}
	
	\subsection{The case when \texorpdfstring{$\mu$}{mu} is not a single column}\label{sec:hooknotcol}
	
	Fix $n \geq 1$, $\mu \vdash n$ with $\mu \neq (1^n)$, and set $\ell = \ell(\mu)$. For $i \in \ZZ_{\geq 1}$, let $m_i = \#\{j \,|\, \mu_j =i \}$, the multiplicity of $i$ in $\mu$. We abbreviate
	\begin{align}
		L := \ell - m_1
	\end{align}
	(the number of parts of $\mu$ that are at least $2$), and use this notation throughout the subsection.
	
	Our goal is to prove that 
	\begin{equation}
		\widehat{b}_{(n-k,1^k),\mu}(q,t) := \sum_{T \in \SSYT((n-k,1^k),\mu)}^{} \psi_{T}(q,t)
	\end{equation} satisfies the recursion in \eqref{eq:brecPieri}. To rewrite the recursion in generating function form, define
	\begin{align}
		\mathcal{R}_\mu(z) := \sum_{\nu: \mu-\nu \cong (1^*)}^{} \psi'_{\mu/\nu}(0,q)\, b_{(|\nu|),\nu}(q,t)\, z^{|\mu-\nu|} \in \ZZ[q,t,z],
	\end{align} where the sum is over all $\nu$ such that $\mu/\nu$ is a vertical strip,
	and
	\begin{align}
		\mathcal{L}_\mu(z) := \sum_{k = 0}^{n-1} \sum_{T \in \SSYT((n-k,1^k),\mu)}^{} \psi_{T}(q,t)z^k \in \ZZ[q,t,z].
	\end{align}
	
	\noindent Then proving \eqref{eq:brecPieri} for $\widehat{b}_{(n-k,1^k),\mu}(q,t)$ is equivalent to proving that
	\begin{align}
		\mathcal{R}_\mu(z) = (1+z) \mathcal{L}_\mu(z).
	\end{align}
	
	The rest of the subsection is a computation of the two sides of this identity, and it is organized as follows. Both sides turn out to have a common factor $q^{n(\mu)}(q^{-1}t;q^{-1})_{\ell-m_1-1}$; after removing it, what remains is an identity between two explicit $q$-binomial generating functions $F$ and $S$ \eqref{eq:FSprod}, which we prove by induction on $m_1$ using matching recurrences (\autoref{lem:Frec} and \autoref{lem:Srec}).
	
	Since $\mu \neq (1^n)$ we have $\mu_1 > 1$, and hence $m = 1$ for every $T \in \SSYT((n-k,1^k),\mu)$; consequently
	\begin{align}
		\psi_{T}(q,t) = q^{\ch(T)} \prod_{x \in D(T)}^{} (1-q^{-\ch_T(x)}t),
	\end{align}
	where 
	\begin{equation}
		D(T) = \{(1,c) \in (n-k,1^k)\,|\, c>1 \hbox{ and } T(1,c-1)<T(1,c)\}
	\end{equation} is the set of boxes in the first row that contribute $(q^{\ch_T(x)}-t)$ in the original formulation.
	
	Throughout the section let
	\begin{equation}
		C = \{T(2,1),\ldots, T(k+1,1)\},
	\end{equation} the set of entries in the first column of $T$ excluding the corner $(1,1)$ box. For $k \in \ZZ$, we denote the entries of $C$ smaller or equal to $k$ by $C_{\leq k} = C \cap \ZZ_{\leq k}$, and similarly for elements of $C$ greater equal $k$ write $C_{\geq k} = C \cap \ZZ_{\geq k}$ etc.
	
	The first step is to make the weights $\psi_T(q,t)$ explicit. In the next two lemmas we rewrite $\psi_T(q,t)$ as a product of $q$-pochhammer symbols, which goes towards simplifying $\cL_\mu(z)$: the boxes where the weight picks up a $(q^{\ch_T(x)}-t)$-factor, and the charges of those boxes, are determined by the column $C$ alone.
	\begin{lemma}\label{lem:psiqthookprod}
		With the notations from above, set
		$\alpha = \#C_{\leq \ell -m_1}$ and $\beta = \# C_{\geq \ell -m_1+1}$.
		Then
		\begin{align}
			\psi_{T}(q,t) = q^{\ch(T)} (q^{-1}t;q^{-1})_{\ell-m_1-1} (q^{-(\ell-m_1-\alpha)}t;q^{-1})_{m_1 - \beta}
		\end{align}
	\end{lemma}

	\begin{proof}
		Suppose $(1,c) \in D(T)$. If $T(1,c) = i$  and $i$ appears more than once in $T$, i.e., $\mu_i >1$, then the $i$ coming from $T(1,c)$ does not appear in the first standard subword, hence $\ch_T((1,c)) = i-1$ in that case. Thus for such boxes, the charge contributions form a consecutive sequence $[1,\ell-m_1-1]$.
		
		Suppose $(1,c) \in D(T)$, $T(1,c) = i$ and $\mu_i = 1$.  Then 
		$$\ch_T((1,c)) = \#\{v \in [2,i] : v \text{ appears in row } 1\} = (i-1) - \#\big([2,i] \cap C\big).$$
		Since $i \in \{\ell-m_1+1,\ldots,\ell\}$, the minimum charge of such boxes is $\ell - m_1 - \alpha$. The charge contributions for these types of boxes again form a consecutive sequence and there are $m_1-\beta$ such boxes.
		
		Thus
		\begin{align*}
			\psi_{T}(q,t) = q^{\ch(T)} (q^{-1}t;q^{-1})_{\ell-m_1-1} (q^{-(\ell-m_1-\alpha)}t;q^{-1})_{m_1 - \beta}. 
		\end{align*}
	\end{proof}
	
	We will also need the total charge of a hook tableau, which again depends only on the column:
	\begin{lemma}\label{lem:chhook}
		With the notations from above,
		\begin{align*}
			\ch(T) = n(\mu) - \ell k + \sum_{e \in C}^{} (e-1).
		\end{align*}
	\end{lemma}
	
	\begin{proof}
		Note that all the entries of $C$ get picked up in the first standard subword; the rest of the standard subwords contribute $n(\widetilde{\mu})$ where $\widetilde{\mu} = (\mu_1-1,\mu_2-1,\ldots)$. The contribution from the first standard subword gives $\binom{\mu_1'}{2} - \sum_{e \in C}^{} (\ell-e+1)$. Since $\binom{\mu_1'}{2} + n(\widetilde{\mu}) = n(\mu)$, we get the result.
	\end{proof}
	
	With the weights in hand, we compute $\cL_\mu(z)$: grouping tableaux by how many singleton values sit in the column turns it into a double $q$-binomial sum, and after pulling out the common factor what remains is the following generating function. Define, for $r \in \ZZ_{\geq 1}$ and $s \in \ZZ_{\geq 0}$,
	\begin{align}\label{eq:Fdef}
		F(r,s,z) := \sum_{\substack{s \geq \beta \geq 0 \\ r-1 \geq \alpha \geq 0}}^{} (q^{-(r-\alpha)}t;q^{-1})_{s - \beta}\, q^{\binom{\alpha+1}{2} - \alpha(r+s) + \binom{\beta}{2} - \beta s}{\qbinom{r-1}{\alpha}}_q {\qbinom{s}{\beta}}_q z^{\alpha+\beta}.
	\end{align}
	
	\begin{lemma}\label{lem:LmuF}
		With the above notations,		
		\begin{align*}
			\cL_\mu(z) =  q^{n(\mu)}(q^{-1}t;q^{-1})_{\ell-m_1-1}
			\cdot F(\ell - m_1,m_1,z).
		\end{align*}
	\end{lemma}
	
	\begin{proof}
		
		For a set $S$ of natural numbers, let $$\Sigma(S) = \sum_{e \in S}^{} (e-1).$$ 
		
		By \autoref{lem:psiqthookprod}, for every $T \in \SSYT((n-k,1^k),\mu)$ the weight $\psi_{T}(q,t)$ is divisible by $(q^{-1}t;q^{-1})_{\ell -m_1-1}$.
		Also $\ch(T) = n(\mu) -\ell k + \sum_{e \in C}^{} (e-1)$ by \autoref{lem:chhook}, where $C$ denotes the column of $T$; hence
		\begin{align*}
			&\widehat{b}_{(n-k,1^k),\mu}(q,t) := \sum_{\substack{T \in \SSYT((n-k,1^k),\mu)}}^{} \psi_{T}(q,t)
			\\
			&= q^{n(\mu) - \ell k} (q^{-1}t;q^{-1})_{\ell-m_1-1} \sum_{T \in \SSYT((n-k,1^k),\mu)}^{} q^{\Sigma(C)} (q^{-(\ell-m_1-\alpha)}t;q^{-1})_{m_1 - \beta},
		\end{align*} where $\beta$ is the number of values in the column of $T$ which do not appear in the row of $T$. Now we can break the sum over $\SSYT((n-k,1^k),\mu)$ by grouping together terms that give the same $\beta = \# C_{\geq \ell - m_1 +1}$.
		Note that $\Sigma(C) = \Sigma(C_{\leq \ell-m_1}) + \Sigma(C_{> \ell-m_1})$ and $\# C_{\leq \ell-m_1} + \# C_{> \ell-m_1} = k$. For a fixed $\beta$, since we can choose a pair consisting of a subset of $\{2,\ldots,\ell-m_1\}$ and a subset of $\{\ell-m_1+1,\ldots,\ell\}$ independently, and this gives a unique $C$ and a corresponding unique $T$, we have, by \eqref{eq:qcountingsubsets},
		\begin{align*}
			\sum_{S \in \binom{[2,\ell-m_1]}{\alpha}} q^{\Sigma(S)} = q^{\binom{\alpha+1}{2}} {\qbinom{\ell-m_1-1}{\alpha}}_q,
		\end{align*}
		and
		\begin{align*}
			\sum_{S \in \binom{[\ell-m_1+1,\ell]}{\beta}} q^{\Sigma(S)} = q^{(\ell-m_1-1)\beta} \cdot q^{\binom{\beta+1}{2}} {\qbinom{m_1}{\beta}}_q.
		\end{align*}
		Therefore,
		\begin{align*}
			&\sum_{\substack{T \in \SSYT((n-k,1^k),\mu) \\ \# C_{\leq \ell -m_1} = \alpha \\ \# C_{> \ell - m_1} = \beta}}^{} q^{\Sigma(C_{\leq \ell-m_1})} q^{\Sigma(C_{> \ell-m_1})}
			= q^{\binom{\alpha+1}{2}}{\qbinom{\ell-m_1-1}{\alpha}}_q q^{(\ell-m_1-1)\beta} q^{\binom{\beta+1}{2}} {\qbinom{m_1}{\beta}}_q.
		\end{align*}

		Thus, summing over all $\beta$, and using $\alpha+\beta = k$ and the identity
		\begin{align*}
			-\ell k + (\ell - m_1 -1)\beta + \binom{\beta+1}{2} = -\alpha \ell - m_1 \beta + \binom{\beta}{2},
		\end{align*}
		we get the statement.
	\end{proof}
	
	For the right hand side we need the dual bookkeeping: removing a vertical strip from $\mu$ removes a final segment from each column of $\mu$, and the following lemma records the resulting $q$-binomial generating function.
	
	\begin{lemma}\label{lem:qbinlastsets}
		Let $a \leq b$, $a,b \in \ZZ_{\geq 0}$ and $m = b-a$. Let $D_r = [b-r+1,b] \subseteq [a+1,b]$ for $r \in [0,m]$. Then
		\begin{align*}
			\sum_{r=0}^{m} q^{-\sum_{i \in D_r}^{}(i-1)}{} {\qbinom{m}{r}}_q z^r = \prod_{j=a}^{b-1} (1+q^{-j}z).
		\end{align*}
	\end{lemma}
	\begin{proof}
		By definition of the elementary symmetric functions $e_r$, the right hand side is $\sum_r e_r(q^{-a},\dots,q^{-(b-1)})z^r$. Using \cite[Ch. I.2 Ex. 3]{MacMainBook} \begin{align}
			e_r(1,q,\dots,q^{m-1})=q^{\binom{r}{2}}\qbinom{m}{r}_q,
		\end{align} we get 
		\begin{align}\label{eq:qbininv}
			e_r(q^{-a},\dots,q^{-(b-1)}) = q^{-(b-1)r}e_r(1,q,\ldots,q^{m-1}) =  q^{\binom{r+1}{2}-br}\qbinom{m}{r}_q.
		\end{align}
		Using 
		\begin{align*}
			\sum_{i\in D_r}(i-1)=r(b-r)+\binom{r}{2} = br - \binom{r+1}{2}
		\end{align*} we get the result.

%		%%%%%%%%%%%%%%%%%%%%%%%%%%%%%%%%%%%%%
%		Using $\binom{r}{2} - r^2 = -\binom{r+1}{2},$ we get
%		\begin{align*}
%			\sum_{i \in D_r}^{} (i-1) &= (b-r+0) + \cdots + (b-r+r-1) = r(b-r) + \binom{r}{2}
%			\\
%			&= r(b-m) + r m - \binom{r+1}{2}.
%		\end{align*}
%		Then using \eqref{eq:qbinrev} and the $q$-binomial theorem \eqref{eq:qbin}, \eqref{eq:qbininv},
%		\begin{align*}
%			&\sum_{r = 0}^{m} q^{-\sum_{i \in D_r}^{} (i-1)} {\qbinom{m}{r}}_q z^{r} = \sum_{r = 0}^{m} q^{-r(b-m) - r m + \binom{r+1}{2}}{\qbinom{m}{r}}_q z^{r}
%			\\
%			&= \sum_{r = 0}^{m} q^{-r(b-m) - r m + \binom{r+1}{2}+r(m-r)}{\qbinom{m}{r}}_{q^{-1}} z^{r}
%			= \sum_{r = 0}^{m} q^{-r(b-m)  + \binom{r+1}{2} -r^2}{\qbinom{m}{r}}_{q^{-1}} z^{r}
%			\\
%			&= \sum_{r = 0}^{m} q^{ - \binom{r}{2}}{\qbinom{m}{r}}_{q^{-1}} (zq^{-(b-m)})^{r} = \prod_{j=a}^{b-1} (1+q^{-j}z).
%		\end{align*}
%		
	\end{proof}
	
	We now compute the right hand side $\cR_\mu(z)$; using the single-row formula \eqref{eq:b_singlerow_intro} and the vertical strip bookkeeping above, it factors into a product of linear terms times a reduced sum.
	
	Define
	\begin{align}\label{eq:Sdef}
		S(n,m,z) := \sum_{r=0}^{m} {\qbinom{m}{r}}_q (q^{-1}t;q^{-1})_{n + m - r - 1}\, q^{\binom{r+1}{2} - r(n+m)} z^r.
	\end{align}
	
	\begin{lemma}\label{lem:RmuS}
		With the above notations,
		$$ \cR_\mu(z) = q^{n(\mu)}\prod_{j=0}^{L-1}(1+zq^{-j})\cdot S(L,m_1,z),$$ where $L = \ell -m_1$.
	\end{lemma}
	
	\begin{proof}
		Write $\nu = \mu - \sum_{i \in D}^{} \varepsilon_i$, where $D$ is the set of parts removed from $\mu$ and $\varepsilon_i$ are the standard basis vectors of $\ZZ^{\ell(\mu)}$.
		Let $D_i = \{j \in D \mid \mu_j = i\}$.
		Let $r = \# D_1$.
		Note that $D_1 = \{\ell-r+1,\ldots,\ell\}$. The contribution from $D_1$ in $\psi'$ is ${\qbinom{m_1}{r}}_q$.
		Note that $\ell(\nu) = \ell - r$.
		Then, by \eqref{eq:b_singlerow_intro},
		\begin{align*}
			b_{(|\nu|),\nu}(q,t) = q^{n(\nu)} (q^{-1}t;q^{-1})_{\ell - r - 1},
		\end{align*}
		and
		\begin{align*}
			n(\mu) - n(\nu) &= \sum_{i \in D}^{} (i-1) = \sum_{i \in D \setminus D_1}^{} (i-1) + r(\ell-r) + \binom{r}{2}.
		\end{align*}
		
		Then
		\begin{align*}
			\cR_\mu(z) &= q^{n(\mu)} \bigg(\sum_{r=0}^{m_1} {\qbinom{m_1}{r}}_q (q^{-1}t;q^{-1})_{\ell - r - 1} q^{-\binom{r}{2} - r(\ell-r)} z^r\bigg)\\
			&\qquad \cdot \bigg(\sum_{\substack{(r_2,r_3,\ldots) \\ r_i = \# D_i}}^{}q^{-\sum_{i \in D \setminus D_1}^{} (i-1)} \prod_{i\geq 2}^{} {\qbinom{m_i}{r_i}}_q z^{r_2+r_3+\cdots}\bigg)
		\end{align*}
		
		Now, the second factor factors as $\prod_{i\geq 2}^{}  \bigg(\sum_{r_i \geq 0}^{} q^{-\sum_{j \in D_i}^{} (j-1)} {\qbinom{m_i}{r_i}}_q z^{r_i}\bigg)$.
		Note that $D_i = \{\mu'_i-r_i+1,\ldots,\mu'_i\} \subseteq \{\mu'_i-m_i+1,\ldots,\mu'_i\}$, so using \autoref{lem:qbinlastsets} we get
		\begin{align*}
			&\sum_{r_i = 0}^{m_i} q^{-\sum_{j \in D_i}^{} (j-1)} {\qbinom{m_i}{r_i}}_q z^{r_i} = \prod_{j=\mu'_{i+1}}^{\mu_i'-1} (1+q^{-j}z).
		\end{align*}
		
		Then
		\begin{align*}
			\cR_\mu(z)
			&= q^{n(\mu)} \bigg(\sum_{r=0}^{m_1} {\qbinom{m_1}{r}}_q (q^{-1}t;q^{-1})_{\ell - r - 1} q^{\binom{r+1}{2} - r\ell} z^r\bigg) \prod_{j=0}^{\ell-m_1-1} (1+q^{-j}z)
			\\
			&= (1+z)\, q^{n(\mu)} \bigg(\sum_{r=0}^{m_1} {\qbinom{m_1}{r}}_q (q^{-1}t;q^{-1})_{\ell - r - 1} q^{\binom{r+1}{2} - r\ell} z^r\bigg) \prod_{j=1}^{\ell-m_1-1} (1+q^{-j}z),
		\end{align*}
		which, since $L = \ell - m_1$, is the stated form.
	\end{proof}
	
	By \autoref{lem:LmuF} and \autoref{lem:RmuS}, proving $\cR_\mu(z) = (1+z)\cL_\mu(z)$ is equivalent to proving
	\begin{align}\label{eq:FSprod}
		(q^{-1}t;q^{-1})_{L-1} \cdot F(L,m_1,z) = S(L,m_1,z) \cdot \prod_{j=1}^{L-1} (1+q^{-j}z),
	\end{align}
	where we write $L = \ell - m_1$ as above. We prove \eqref{eq:FSprod} by induction on $m_1$: the base case $m_1 = 0$ is the $q$-binomial theorem, and the induction step applies the same $q$-Pascal identity to both sides, producing the following two recurrences, which are matched in \autoref{prop:LR}. First, the recurrence on the tableau side is as follows.
	
	\begin{lemma}\label{lem:Frec}
    For $n \in \ZZ_{\geq 1}$ and $m \in \ZZ_{\geq 0}$,
		\begin{align*}
			F(n,m+1,z) = (1+q^{-1}z)F(n,m,q^{-1}z) - q^{-(n+m)}t F(n,m,z)
		\end{align*}
	\end{lemma}
	
	\begin{proof}
		Recall from \eqref{eq:Fdef} that
		\begin{align}
			F(n,m,z) := \sum_{\substack{m \geq \beta \geq 0 \\ n-1 \geq \alpha \geq 0}}^{} (q^{-(n-\alpha)}t;q^{-1})_{m - \beta}\, q^{\binom{\alpha+1}{2} - \alpha(n+m) + \binom{\beta}{2} - \beta m}{\qbinom{n-1}{\alpha}}_q {\qbinom{m}{\beta}}_q z^{\alpha+\beta}.
		\end{align}
		Using the $q$-Pascal identity \eqref{eq:qPascal} 
		\begin{align*}
			{\qbinom{m+1}{\beta}}_q = {\qbinom{m}{\beta}}_q + q^{m-\beta+1}{\qbinom{m}{\beta-1}}_q,
		\end{align*} we get
		\begin{align*}
			F(n,m+1,z) = F_1 + F_2
		\end{align*}
		where since ${m \brack m+1}_q = 0$, the first term is
		\begin{align*}
			&F_1 = \sum_{\substack{m \geq \beta \geq 0 \\ n-1 \geq \alpha \geq 0}}^{} (q^{-(n-\alpha)}t;q^{-1})_{m+1 - \beta} q^{\binom{\alpha+1}{2} - \alpha(n+m+1) + \binom{\beta}{2} - \beta (m+1)}{\qbinom{n-1}{\alpha}}_q {\qbinom{m}{\beta}}_q z^{\alpha+\beta}.
        \end{align*}
		% 	\\
		% 	&= \sum_{\substack{m \geq \beta \geq 0 \\ n-1 \geq \alpha \geq 0}}^{} {\qbinom{m}{\beta}}_q (1-q^{-(n-\alpha+m-\beta)}t) (q^{-(n-\alpha)}t;q^{-1})_{m - \beta}\cdot q^{\binom{\alpha+1}{2} - \alpha(n+m) + \binom{\beta}{2} - \beta m -\alpha-\beta} {\qbinom{n-1}{\alpha}}_q z^{\alpha+\beta}
		% 	\\
		% 	&= F(n,m,q^{-1}z) - q^{-(n+m)}t F(n,m,z),
		% \end{align*}
        Using $(a;q^{-1})_{m+1-\beta}=(1-aq^{-(m-\beta)})(a;q^{-1})_{m-\beta}$ with $a=q^{-(n-\alpha)}t$, we get that 
        \begin{align*}
            F_1 = F(n,m,q^{-1}z) - q^{-(n+m)}t F(n,m,z),
        \end{align*}
		and in the second term shift $\beta \mapsto \beta+1$ to get
        \begin{align*}
            F_2 = q^{-1}z F(n,m,q^{-1}z).
        \end{align*}
		% \begin{align*}
		% 	F_2 &= \sum_{\substack{m+1 \geq \beta \geq 0 \\ n-1 \geq \alpha \geq 0}}^{} q^{m-\beta+1+\binom{\alpha+1}{2} - \alpha(n+m) + \binom{\beta}{2} - \beta m -\alpha-\beta}{\qbinom{m}{\beta-1}}_q  (q^{-(n-\alpha)}t;q^{-1})_{m +1 - \beta}  {\qbinom{n-1}{\alpha}}_q z^{\alpha+\beta}
		% 	\\
		% 	&= q^{-1}z F(L,m,q^{-1}z)
		% \end{align*} by using the shift $\beta \mapsto \beta+1$.
	\end{proof}
	
	The recurrence on the right hand side, proved by the same computation is as follows.
	
	\begin{lemma}\label{lem:Srec}
    For $n \in \ZZ_{\geq 1}$ and $m \in \ZZ_{\geq 0}$,
		\begin{align*}
			S(n,m+1,z) = (1+zq^{-n})S(n,m,q^{-1}z) - q^{-(n+m)} t S(n,m,z).
		\end{align*}
	\end{lemma}
	\begin{proof}
		Recall from \eqref{eq:Sdef} that
		\begin{align}
			S(n,m,z) := \sum_{r=0}^{m} {\qbinom{m}{r}}_q (q^{-1}t;q^{-1})_{n + m - r - 1}\, q^{\binom{r+1}{2} - r(n+m)} z^r.
		\end{align}
		Using the same $q$-Pascal identity as in \autoref{lem:Frec},
        \begin{align*}
            S(n,m+1,z) &= \sum_{r=0}^{m+1} {\qbinom{m+1}{r}}_q (q^{-1}t;q^{-1})_{n + m - r}\, q^{\binom{r+1}{2} - r(n+m+1)} z^r
            \\
            &= \sum_{r=0}^{m} {\qbinom{m}{r}}_q (q^{-1}t;q^{-1})_{n + m - r}\, q^{\binom{r+1}{2} - r(n+m+1)} z^r
            \\
            &\qquad + \sum_{r=1}^{m+1} q^{m+1-r}{\qbinom{m}{r-1}}_q (q^{-1}t;q^{-1})_{n + m - r}\, q^{\binom{r+1}{2} - r(n+m+1)} z^r
        \end{align*}
        For $0\le r\le m$ we have $n+m-r\ge n\ge 1$, so we may write
        $$(q^{-1}t;q^{-1})_{n+m-r}=(1-q^{-(n+m-r)}t)(q^{-1}t;q^{-1})_{n+m-r-1}.$$ In the second sum set $s=r-1$. Then        
		\begin{align*}
			S(n,m+1,z) &= \sum_{r=0}^{m} {\qbinom{m}{r}}_q (1-q^{-(n+m-r)}t)(q^{-1}t;q^{-1})_{n+m-r-1} q^{\binom{r+1}{2} - r(n+m+1)} z^r
            \\
            &\qquad + \sum_{s=0}^{m} q^{m-s}{\qbinom{m}{s}}_q (q^{-1}t;q^{-1})_{n + m - s-1}\, q^{\binom{s+2}{2} - (s+1)(n+m+1)} z^{s+1}
            \\
            &= S(n,m,q^{-1}z) - q^{-(n+m)}t S(n,m,z) + zq^{-n} S(n,m,q^{-1}z).
		\end{align*}
		% The first summand gives $S(L,m,q^{-1}z) - q^{-(L+m)}t S(L,m,z)$, and the second summand is $zq^{-L} S(L,m,q^{-1}z)$, since
		% \begin{align*}
		% 	m-s+\binom{s+2}{2} - (s+1)(L+m+1) &= \binom{s+1}{2}- s (L+m)  - L - s. 
		% \end{align*}
  %       This completes the proof.
	\end{proof}
	
	Now the two recurrences assemble into the desired identity.
	
	\begin{proposition}\label{prop:LR}
		\begin{align*}
			(1+z)\cL_\mu(z) = \cR_\mu(z).
		\end{align*}
	\end{proposition}
	\begin{proof}
		The statement is equivalent to proving \eqref{eq:FSprod}. When $m = 0$, by \eqref{eq:qbininv},
		\begin{align*}
			\prod_{j=1}^{L-1}(1+q^{-j}z) = \sum_{j \geq 0}^{} q^{\binom{j}{2}-j(L-1)}{\qbinom{L-1}{j}}_q z^j = \sum_{j=0}^{L-1} q^{\binom{j+1}{2}-jL} {\qbinom{L-1}{j}}_q  z^j,
		\end{align*}
		so both sides of \eqref{eq:FSprod} equal $(q^{-1}t;q^{-1})_{L-1} \cdot \prod_{j=1}^{L-1}(1+q^{-j}z)$. For the induction step,
		\begin{align*}
			&(q^{-1}t;q^{-1})_{L-1} \cdot F(L,m+1,z)\\
			&=  (q^{-1}t;q^{-1})_{L-1} \cdot \big((1+q^{-1}z)F(L,m,q^{-1}z) - q^{-(L+m)}t F(L,m,z)\big)
			\\
			&=  (1+q^{-1}z)\prod_{j=2}^{L} (1+q^{-j}z)\, S(L,m,q^{-1}z)-  q^{-(L+m)}t \prod_{j=1}^{L-1} (1+q^{-j}z)\, S(L,m,z)
			\\
			&= \prod_{j=1}^{L-1} (1+q^{-j}z) \cdot \bigg((1+q^{-L}z)S(L,m,q^{-1}z) - q^{-(L+m)}t S(L,m,z)\bigg)
			\\
			&= \prod_{j=1}^{L-1} (1+q^{-j}z) \cdot S(L,m+1,z),
		\end{align*}
		where the second equality uses the induction hypothesis and the last equality uses \autoref{lem:Srec}.
	\end{proof}
	
	This completes the proof of \autoref{thm:mainlocal} when $\mu$ is not a single column.
	
	\subsection{The case when \texorpdfstring{$\mu$}{mu} is a single column}\label{sec:hooksinglecol}
	
	When $\mu = (1^n)$ the situation is different in two ways: all entries of $T$ are distinct, so the factor $[m]_t$ is no longer trivial and \autoref{lem:psiqthookprod} does not apply; on the other hand, the recursion \eqref{eq:brecPieri} has only one term on the right-hand side, and it is simple enough to solve explicitly.
	
	Define
	\begin{align*}
		C_{n,k} := b_{(n-k,1^k),(1^{n})}(q,t)
	\end{align*}
	and
	\begin{align}\label{eq:gamman}
		\gamma_n := q^{\binom{n}{2}}(q^{-1}t;q^{-1})_{n-1} = \prod_{i=1}^{n-1}(q^i-t).
	\end{align}
	
	Specializing \eqref{eq:brecPieri} to a single column gives the recursion
	\begin{align*}
		C_{n,k} + C_{n,k-1} = \gamma_{n-k} \cdot {\qbinom{n}{k}}_q \qquad \hbox{ for } 1\leq k \leq n-2,
	\end{align*}
	with initial value $C_{n,0} = \gamma_n$. This recursion has a unique solution, which we can write down:
	
	\begin{proposition}\label{prop:Cnk}
		For $n \geq 2$ and $n-1 > k \geq 0$,
		$$ C_{n,k} = \gamma_{n-k} \cdot \sum_{i=0}^{k} {\qbinom{n-1-i}{n-k-1}}_q t^i.$$
	\end{proposition}
	\begin{proof}
		For $k = 0$ both sides equal $\gamma_n$. Hence, by induction on $k$, it suffices to check that the claimed formula satisfies the same recursion as $C_{n,k}$. Indeed, assuming the formula for $(n,k)$ and $(n,k-1)$,
		\begin{align*}
			&C_{n,k} + C_{n,k-1} = \gamma_{n-k} \bigg( \sum_{i=0}^{k} {\qbinom{n-1-i}{n-k-1}}_q t^i + (q^{n-k}-t) \cdot \sum_{i=0}^{k-1} {\qbinom{n-1-i}{n-k}}_q t^i\bigg)
			\\
			&= \gamma_{n-k} \bigg( \sum_{i=0}^{k-1} \bigg({\qbinom{n-1-i}{n-k-1}}_q +q^{n-k} {\qbinom{n-1-i}{n-k}}_q\bigg)t^i + t^k - \sum_{i=0}^{k-1} {\qbinom{n-1-i}{n-k}}_q t^{i+1}\bigg)
			\\
			&= \gamma_{n-k} \bigg( \sum_{i=0}^{k-1} {\qbinom{n-i}{n-k}}_q t^i  + t^k - \sum_{i=0}^{k-1} {\qbinom{n-1-i}{n-k}}_q t^{i+1}\bigg)\\
			&= \gamma_{n-k} {\qbinom{n}{k}}_q.
		\end{align*}
	\end{proof}
	
	It remains to show that the sum $\sum_{T \in \SSYT((n-k,1^k),(1^{n}))}^{} \psi_{T}(q,t)$ equals the formula of \autoref{prop:Cnk}. The content of $T$ being $(1^{n})$, the entries of $T$ are pairwise distinct, so the first row is strictly increasing and $T(1,1) = 1$; write
	$$R = \{T(1,2),\ldots, T(1,n-k)\} \subseteq \{2,\ldots,n\}$$
	for the remaining entries of the first row, and $r_1 < \cdots < r_{n-k-1}$ for the elements of $R$ in increasing order. Then $m = r_1 - 1$: the values $2,\ldots,r_1-1$ cannot appear in the first row (which is strictly increasing with smallest entry $r_1$ after $1$), so they occupy the beginning of the first column in order, and the run stops there because $r_1$ itself is in the first row. Also, every box $(1,c)$ with $c \geq 2$ lies in $D(T)$. Hence, by \S\ref{sec:charge},
	\begin{align*}
		\psi_{T}(q,t) = [r_1-1]_t \cdot (q^{-1}t;q^{-1})_{n-k-1}\, q^{\binom{n-k}{2}} \cdot q^{\sum_{x \in R}^{}(n+1-x) - \binom{n-k}{2}},
	\end{align*}
	so
	\begin{align*}
		\sum_{T \in \SSYT((n-k,1^k),(1^{n}))}^{} \psi_{T}(q,t)
		&= \gamma_{n-k} \cdot \sum_{R \in \binom{[2,n]}{n-k-1}}^{} [r_1-1]_t\, q^{\sum_{x \in R}^{}(n+1-x) - \binom{n-k}{2}}
		\\
		&= \gamma_{n-k} \cdot \sum_{u \geq 0}^{} t^u {\qbinom{n-u-1}{n-k-1}}_q = C_{n,k},
	\end{align*}
	where the middle equality follows by writing $[r_1-1]_t = \sum_{u \leq r_1-2} t^u$ and using \eqref{eq:qcountingsubsets}, and the last equality is \autoref{prop:Cnk}.
	
	This completes the proof of \autoref{thm:mainlocal}.

	\section{Two-column shapes}\label{sec:twocol}
	
	In this section we prove \autoref{thm:twocol-intro}, which says that \autoref{conj:B} holds when $\lambda_1 \leq 2$.
	
	By \autoref{prop:reductionprinc}, for $r,s,m,n \in \ZZ_{\geq 0}$, with $r \geq m$, $$ b_{(2^r1^s),(2^m1^n)}(q,t) = b_{(2^{r-m}1^s),(1^n)}(q,t),$$ hence to prove that $b_{\lambda,\mu}(t+1,t) \in \ZZ_{\geq0}[t]$ when $\lambda_1 \leq 2$, it is enough to show that $b_{(2^r1^s),(1^{2r+s})}(t+1,t) \in \ZZ_{\geq 0}[t]$. We first compute the coefficient $b_{(2^r1^s),(1^N)}(q,t)$ explicitly as an alternating sum (\autoref{lem:twocolumn}). Next, we pair consecutive terms in the alternating sum and each pair is rewritten as a visibly positive part plus a remainder, and the remainder is shown to be positive at $q=t+1$ in \autoref{lem:Tpos}.
	
	Let $$D^{(N)}_j(u)={\qbinom{N}{j}}_u-{\qbinom{N}{j-1}}_u.$$

	\begin{lemma}\label{lem:twocolumn}
	Let $r,s\in\ZZ_{\ge0}$ and put $N=2r+s$.  Then
	\begin{align*}
		b_{(2^r1^s),(1^N)}(q,t)
		=\sum_{j=0}^{r}(-1)^{r-j}\,t^{\binom{r-j+1}{2}}\,
		D_j^{(N)}(q)\,
		{\qbinom{s+r-j}{r-j}}_t .
	\end{align*}
\end{lemma}

\begin{proof}
	By \cite[pg.~11]{KRTW_remarks} 
	\begin{align}\label{eq:KRTW-rs}
		P_{(2^r,1^s)}(t)
		=\sum_{i=0}^{r}(-1)^i t^{\binom{i}{2}}\,
		\frac{1-t^{\,2i+s}}{1-t^{\,i+s}}\,
		{\qbinom{i+s}{s}}_t\;e_{r-i}\,e_{r+s+i},
	\end{align} where if $i=s=0$, the corresponding summand is $e_r^2$.
	Re-index by $j=r-i$ gives
	\begin{align}\label{eq:KRTW-reindexed-rs}
		P_{(2^r,1^s)}(t)
		=\sum_{j=0}^{r}(-1)^{r-j} t^{\binom{r-j}{2}}\,
		\frac{1-t^{\,N-2j}}{1-t^{\,s+r-j}}\,
		{\qbinom{s+r-j}{r-j}}_t\;e_{j}\,e_{N-j}.
	\end{align}

	From \eqref{eq:blam1nprinc}
	\[
	b_{(2^r1^s),(1^N)}(q,t)
	=(q;q)_N\Bigl(\omega P_{(2^r,1^s)}(t)\Bigr)\!\left[\frac{1}{1-q}\right].
	\]
	We know \cite[Ch. I.2 Ex. 4]{MacMainBook}, $$h_n\!\left[\frac{1}{1-q}\right]=\frac{1}{(q;q)_n}.$$ Thus,
	\begin{align}\label{eq:raw-formula-rs}
		b_{(2^r1^s),(1^N)}(q,t)
		=\sum_{j=0}^{r}(-1)^{r-j} t^{\binom{r-j}{2}}\,
		\frac{1-t^{\,N-2j}}{1-t^{\,s+r-j}}\,
		{\qbinom{s+r-j}{r-j}}_t\,
		{\qbinom{N}{j}}_q.
	\end{align}
	Denote the coefficients in \eqref{eq:KRTW-reindexed-rs} by $F_j(t)$, i.e.,
	$$ F_j(t) := (-1)^{r-j} t^{\binom{r-j}{2}}\,
	\frac{1-t^{\,N-2j}}{1-t^{\,s+r-j}}\,
	{\qbinom{s+r-j}{r-j}}_t.$$
	
	\noindent	
	Define
	\begin{align}\label{eq:Gj-rs}
		G_j(t):=(-1)^{r-j}\,t^{\binom{r-j+1}{2}}\,
		{\qbinom{s+r-j}{r-j}}_t ,
		\qquad\text{for}\qquad 0\le j\le r,
	\end{align} and $G_{r+1}(t):=0$.
	We claim that
	\begin{align}\label{eq:telescope-rs}
		F_j(t)=G_j(t)-G_{j+1}(t)
		\qquad\text{for }\qquad0\le j\le r.
	\end{align}
	
	To prove \eqref{eq:telescope-rs}, divide both sides by $(-1)^{r-j}t^{\binom{r-j}{2}}{\qbinom{s+r-j-1}{r-j-1}}_t$ (valid for $j<r$; the case $j=r$ is immediate since $F_r=1=G_r$).  Using
	\[
	\frac{{\qbinom{s+r-j}{r-j}}_t}{{\qbinom{s+r-j-1}{r-j-1}}_t}
	=\frac{1-t^{\,s+r-j}}{1-t^{\,r-j}},
	\]
	the left--hand side of \eqref{eq:telescope-rs} becomes
	\[
	\frac{1-t^{\,N-2j}}{1-t^{\,s+r-j}}\cdot\frac{1-t^{\,s+r-j}}{1-t^{\,r-j}}
	=\frac{1-t^{\,N-2j}}{1-t^{\,r-j}} .
	\]
	The right--hand side becomes
	\[
	t^{\,r-j}\cdot\frac{1-t^{\,s+r-j}}{1-t^{\,r-j}}+1
	=\frac{t^{\,r-j}-t^{N-2j}+1-t^{\,r-j}}{1-t^{\,r-j}}
	=\frac{1-t^{\,N-2j}}{1-t^{\,r-j}}.
	\]
	This proves \eqref{eq:telescope-rs}.
	
	\noindent		
	Inserting \eqref{eq:telescope-rs} into  \eqref{eq:raw-formula-rs} gives
	\begin{align*}
		b_{(2^r1^s),(1^N)}(q,t)
		&=\sum_{j=0}^{r}\bigl(G_j(t)-G_{j+1}(t)\bigr)\,{\qbinom{N}{j}}_q 
		\\
		&=\sum_{j=0}^{r}G_j(t)\,{\qbinom{N}{j}}_q
		-\sum_{j=1}^{r+1}G_{j}(t)\,{\qbinom{N}{j-1}}_q .
	\end{align*}
	Since $G_{r+1}(t)=0$, we get
	\[
	b_{(2^r1^s),(1^N)}(q,t)
	=\sum_{j=0}^{r}G_j(t)\left({\qbinom{N}{j}}_q-{\qbinom{N}{j-1}}_q\right).
	\]
	Substituting the definition \eqref{eq:Gj-rs} of $G_j(t)$ yields the claimed formula.
\end{proof}
	
	Note that $\binom{r-j+2}{2} - \binom{r-j+1}{2} = r-j+1$. Pairing consecutive terms from the top, with $j \equiv r \ (\mathrm{mod}\ 2)$ and $m := r-j+1$, we are led to consider
	$$B_j:=D^{(N)}_j(q)\,{\qbinom{s+m-1}{m-1}}_t-t^m\,D^{(N)}_{j-1}(q)\,{\qbinom{s+m}{m}}_t.$$
	To prove $b_{(2^r1^s),(1^N)}(t+1,t) \in \ZZ_{\geq 0}[t]$, we will show that each $B_j|_{q=t+1} \in \ZZ_{\geq 0}[t]$. Note that if $r$ is odd, all terms of \autoref{lem:twocolumn} are paired, and if $r$ is even, the unpaired $j=0$ term is $t^{\binom{r+1}{2}} D_0^{(N)}(q) \qbinom{r+s}{r}_t = t^{\binom{r+1}{2}} \qbinom{r+s}{r}_t \in \ZZ_{\geq 0}[t].$
	
	We can write
	\begin{align*}
		B_j=t^m{\qbinom{s+m}{m}}_t{\qbinom{N}{j-2}}_q+E_j,
	\end{align*}
	where
	\begin{align*}
		E_j={\qbinom{s+m-1}{m-1}}_t{\qbinom{N}{j}}_q-\bigg({\qbinom{s+m-1}{m-1}}_t+t^m{\qbinom{s+m}{m}}_t\bigg){\qbinom{N}{j-1}}_q,
	\end{align*} and note that  \begin{align*}
	B_j-E_j =
		t^m{\qbinom{s+m}{m}}_t{\qbinom{N}{j-2}}_q \in \ZZ_{\geq 0}[q,t].
	\end{align*}
	Thus, to show that $B_j|_{q=t+1} \in \ZZ_{\geq 0}[t]$, it is enough to show that each $E_j|_{q=t+1} \in \ZZ_{\geq 0}[t]$.
	
	Recall from \eqref{eq:qcountingsubsets} that
	$${\qbinom{s+m-1}{m-1}}_t=\sum_{S\in\binom{[s+m-1]}{m-1}}t^{\sigma(S)}, \qquad \hbox{ where } \qquad  \sigma(S)=\sum_{s\in S}s-\binom{\#S+1}{2}.$$
	There is a bijection
	\begin{align*}
		\Bigl\{(S,e)\;\Big|\; S \in \binom{[s+m-1]}{m-1},\ \max S<e\le s+m\Bigr\} &\longleftrightarrow \binom{[s+m]}{m},\\
		(S,e) &\mapsto S\cup\{e\},
	\end{align*}	
	and $m+\sigma(S\cup\{e\})-\sigma(S)=e$. Therefore,
	\begin{align}\label{eq:Ej=sumTj}
		E_j=\sum_{S\in\binom{[s+m-1]}{m-1}}t^{\sigma(S)}\,T\big(N,j,\{M_S{+}1,\dots,s{+}m\}\big),
	\end{align} where 
	\begin{align}
		T(N,j,E):={\qbinom{N}{j}}_q-{\qbinom{N}{j-1}}_q\Big(1+\textstyle\sum_{e\in E}t^e\Big).
	\end{align} and $M_S = \max S$ (with the convention that $E = \{1,\ldots,s+m\}$ when $S$ is empty).
	
	To show that $E_j|_{q=t+1} \in \ZZ_{\geq 0}[t]$ it is enough to show that each summand in \eqref{eq:Ej=sumTj} is positive at $q=t+1$, which is the next lemma.
	
	\begin{lemma}\label{lem:Tpos}
		Suppose $N \geq 2j$ and $L \in \ZZ_{\geq 0}$ with $L \leq N-2j+1$. Then for any $E \subseteq [1,L]$, $T(N,j,E)|_{q=t+1} \in \ZZ_{\geq 0}[t]$.
	\end{lemma}
	
	\begin{proof}
		Put $A_E(t) = 1 + \sum_{e \in E}^{} t^e$. Then
		\begin{align*}
			T(N,j,E) = \bigg({\qbinom{N}{j}}_q - q^{N-2j+1}{\qbinom{N}{j-1}}_q\bigg) + \big(q^{N-2j+1}-A_E(t)\big){\qbinom{N}{j-1}}_q.
		\end{align*}
		
		Since $E \subseteq [1,L]$, and $N-2j+1 \geq L$, $(1+t)^{N-2j+1} - A_E(t) \in \ZZ_{\geq 0}[t]$. Thus the second term is in $\ZZ_{\geq0}[t]$ at $q=t+1$. The first term is
		\begin{align*}
			C_{N,j}(q) := {\qbinom{N}{j}}_q - q^{N-2j+1}{\qbinom{N}{j-1}}_q = \dfrac{1-q^{N-2j+1}}{1-q^{N-j+1}} {\qbinom{N}{j}}_q
			= \dfrac{[M-j+1]_q}{[M+1]_q} {\qbinom{M+j}{j}}_q,
		\end{align*}
		where $M = N-j \geq j$. 
		
		For a tableau $T$ with content $(1^N)$, the descent set $\mathrm{Des}(T)$ is the set of $i$ such that $i+1$ appears in a row strictly below $i$ in $T$, and $\maj(T) = \sum_{i \in \mathrm{Des}(T)} i$. The $q$-hook formula \cite[Corollary 7.21.5]{EC2} gives
		\begin{align*}
			\sum_{T \in \SSYT((M,j),(1^{M+j}))}^{} q^{\maj(T)} = q^j \dfrac{(q;q)_{M+j}}{(q;q)_j (q;q)_{M-j} (q^{M-j+2};q)_j} = q^j \dfrac{[M-j+1]_q}{[M+1]_q} {\qbinom{M+j}{j}}_q,
		\end{align*} where $\maj$ is the major index statistic on tableaux. Note that the minimum possible $\maj$ in this case is $j$, for the tableau 
		\begin{align*}
			T_{\min} = 
			\ytableausetup{boxsize=3em}
			\begin{ytableau}
				1 & 2 & \cdots & j & 2j+1 & \cdots & M+j \\
				j+1 & j+2 & \cdots & 2j
			\end{ytableau}
		\end{align*}  This implies $C_{N,j}(q) \in \ZZ_{\geq 0}[q]$, and hence $C_{N,j}(t+1) \in \ZZ_{\geq 0}[t]$.
	\end{proof}

Thus we proved the following theorem.	
	\begin{theorem}\label{thm:twocol}
		Let $\lambda,\mu \in \Par$ with $\lambda_1 \leq 2$. Then $b_{\lambda,\mu}(t+1,t) \in \ZZ_{\geq 0}[t]$.
	\end{theorem}
	
	\begin{proof}
		By the reduction at the beginning of this section, we may assume $\lambda = (2^r1^s)$ and $\mu = (1^N)$ with $N = 2r+s$. Write the formula of \autoref{lem:twocolumn} as
		\begin{align*}
			b_{(2^r1^s),(1^N)}(q,t)=\sum_{j=0}^r(-1)^{r-j}\,t^{\binom{r-j+1}{2}}\,D^{(N)}_j(q)\,{\qbinom{r+s-j}{r-j}}_t .
		\end{align*}
		Pair the terms $j$ and $j-1$ for $j = r, r-2, r-4, \ldots$. The pair indexed by $j$ equals $t^{\binom{r-j+1}{2}} B_j$. By the discussion above, $B_j = t^m \qbinom{s+m}{m}_t \qbinom{N}{j-2}_q + E_j$, where the first term is in $\ZZ_{\geq0}[q,t]$ and, at $q = t+1$,
		\begin{align*}
			E_j|_{q=t+1} = \sum_{S \in \binom{[s+m-1]}{m-1}} t^{\sigma(S)} T\big(N,j,\{M_S+1,\ldots,s+m\}\big)\big|_{q=t+1} \in \ZZ_{\geq 0}[t]
		\end{align*}
		by \autoref{lem:Tpos} applied with $L = s+m$ (note $N - 2j + 1 = 2r+s-2j+1 \geq r-j+1+s = L$ since $j \leq r$). If $r$ is odd, all terms are paired and we are done. If $r$ is even, the unpaired $j=0$ term is $t^{\binom{r+1}{2}} D_0^{(N)}(q) \qbinom{r+s}{r}_t = t^{\binom{r+1}{2}} \qbinom{r+s}{r}_t \in \ZZ_{\geq 0}[t]$, completing the proof.
	\end{proof}

	%%%%%%%%%%%%%%%%%%%%%%%%%%%%%%%%%%%%%%%%%%%%%%%%%%
	\section{Some small cases}\label{sec:smallcases}
	
	In this section we record some coefficients $b_{\lambda,\mu}(q,t)$ for small shapes. All of them can be computed by using the monomial expansion formula and the triangularity of $P$ to peel off one coefficient at a time.
	
	By \autoref{rem:stabilityfinitecheck}, infinite families in the stable range $\mu_1 > \lambda_2$ reduce to a single computation. For instance, item (1) of \autoref{prop:smallcases} below for all $n \geq k+2$ is the single computation $b_{(k+2,k),(k+1,k,1)}(q,t) = (q-t)^2 + (q-t)$, and similarly for the other items. We skip the proofs which can be obtained as outlined above, the main purpose of these examples is to motivate the discussion on the possibility of term-by-term weights.

	\begin{proposition}\label{prop:smallcases}
		\begin{enumerate}
			\item 	For $n \geq k+2$, $k \geq 2$, \[ b_{(n,k),(n-1,k,1)}(q,t) = (q-t)^2+(q-t).\]
			
			\item 	For $n \geq k +2$, $k \geq 3$, \[ b_{(n,k,1),(n-1,k,1,1)}(q,t) = q(q-t)^2+(q-t)^2 + q(q-t)+ (q-t). \]
		\end{enumerate}
	\end{proposition}
	
	\begin{remark}
		The hypothesis $k \geq 3$ in item (2) is necessary: for $k=2$ one has $b_{(n,2,1),(n-1,2,1,1)}(q,t) = (q-t)(q^2+2q-t^2-t+1)$, which does not factor as in item (2).
	\end{remark}
	
	\begin{remark}[Term-by-term weights]\label{rem:termbiterm}
		The formulas of \autoref{prop:smallcases} are sums over tableaux, term by term, with the box weight of \autoref{thm:main}: for each $T \in \SSYT(\lambda,\mu)$, with
		\begin{align*}
			\psi_T(q,t) = \prod_{\substack{x \in \lambda,\, \coarm(x) > 0 \\ T(\Left(x)) < T(x)}} \big(q^{\ch_T(x)} - t\big) \prod_{\text{other } x} q^{\ch_T(x)},
		\end{align*}
		one has $b_{\lambda,\mu}(q,t) = \sum_{T} \psi_T(q,t)$, and moreover each tableau satisfies $\psi_T(q,0) = q^{\ch(T)}$ and $\psi_T(1,t) = \psi_T(t)$ individually. For instance, in item (1) the two tableaux in $\SSYT((k+2,k),(k+1,k,1))$ contribute $(q-t)^2$ and $(q-t)$, and in item (2) the four tableaux in $\SSYT((k+2,k,1),(k+1,k,1,1))$ contribute $q(q-t)^2$, $(q-t)^2$, $q(q-t)$ and $(q-t)$. This raises the question whether $\psi_T$ can be generalized to give a tableaux formula for $b_{\lambda,\mu}(q,t)$ for general shapes.
	\end{remark}
	
	By \autoref{Bdef}, one can compute $b_{\lambda,\mu}(q,t)$ as transition coefficients when expanding $Q'_\mu(q)$ in the $Q'(t)$-basis. In terms of $Q'$ one can show, using \autoref{prop:stability} and \autoref{prop:reductionprinc},
	\begin{proposition}\label{prop:Qpexpansions}
		\begin{enumerate}
			\item For $n \geq  4$,
			\begin{align*}
				Q'_{(n-2,1^2)}(q) &= (q^2-t)(q-t)Q'_{(n)}(t) + (q-t)(q+1)Q'_{(n-1,1)}(t)
				\\
				&\qquad + (q-t)Q'_{(n-2,2)}(t) + Q'_{(n-2,1^2)}(t).
			\end{align*}
			
			\item For $n \geq 6$,
			\begin{align*}
				Q'_{(n-3,1^3)}(q) &= (q^3-t)(q^2-t)(q-t)Q'_{(n)}(t) + (q^2+q+1)(q^2-t)(q-t)Q'_{(n-1,1)}(t)
				\\
				&\qquad+ (q^2+q+1)(q-t)^2 Q'_{(n-2,2)}(t) + (q^2+q+1)(q-t)Q'_{(n-2,1,1)}(t)
				\\
				&\qquad+ (q^2-t)(q-t) Q'_{(n-3,3)}(t) + (q+t+1)(q-t)Q'_{(n-3,2,1)}(t)\\
				&\qquad+ Q'_{(n-3,1^3)}(t).
			\end{align*}
			
			\item For $n \geq 3$,
			\begin{align*}
				Q'_{(n,2,1)}(q) &= q(q^2-t)(q-t) Q'_{(n+3)}(t) + (q^2+q-t)(q-t)Q'_{(n+2,1)}(t)
				\\
				&\qquad+ (q-t+1)(q-t) Q'_{(n+1,2)}(t) + (q-t) Q'_{(n+1,1,1)}(t)
				\\
				&\qquad+ (q-t)Q'_{(n,3)}(t) + Q'_{(n,2,1)}(t).
			\end{align*}
			
		\end{enumerate}
	\end{proposition}

	\begin{remark}\label{rem:Qptermbiterm}
		The coefficients in \autoref{prop:Qpexpansions} are also tableaux sums term by term. In items (1) and (3), the plain box weight of \autoref{rem:termbiterm} works for every $\lambda$. In item (2), the box weight works for every $\lambda$ except two coefficients: for $\lambda = (n-3,2,1)$ and $\lambda = (n-2,2)$. For $\lambda = (n-3,2,1)$ we can set the weights as 
		\ytableausetup{smalltableaux}
		\begin{align*}
			\psi_{\ytableaushort{11\cdots 1,23,4}}(q,t) = q(q-t), \qquad
			\psi_{\ytableaushort{11\cdots 1,24,3}}(q,t) = (q-t)(1+t),
		\end{align*} this is consistent with our hook tableaux formula \eqref{eq:psiqtdef} after disregarding the top row of all $1$. For $\lambda = (n-2,2)$ we can set the weights to be
		\begin{align*}
			\psi_{\ytableaushort{11\cdots 13,24}}(q,t) = (q-t)^2, 
			\\
			\psi_{\ytableaushort{11\cdots 14,23}}(q,t) = q(q-t)^2, 
			\\
			\psi_{\ytableaushort{11\cdots 12,34}}(q,t) = q^2(q-t)^2.
		\end{align*} The three tableaux, of charges $2,3,4$, contribute $q^{i}(q-t)^2$ for $i = 0,1,2$ respectively. We do not know a uniform weight rule that covers all of these cases simultaneously.
	\end{remark}
	
	Item (3) is the case $k=2$ of \autoref{thm:nk1}.
	
	\section{The expansion of \texorpdfstring{$Q'_{(n,k,1)}(q)$}{Q'(n,k,1)(q)}}\label{sec:nk1}
	
	In this section we prove a formula for the transition coefficients $b_{\lambda,(n,k,1)}(q,t)$ in the stable range. The proof is by using the Kostant partition formula \eqref{eq:Q'KostantPartitionexp}. The formula generalizes item (3) of \autoref{prop:Qpexpansions}, which is the case $k=2$.
	
	Note that the expansion of $Q'_{\mu}(q)$ in terms of $Q'_\lambda(t)$ is done for the case when $\mu$ has at most two rows, since if $\mu$ is 2-row and $\lambda \geq \mu$ then $\lambda$ also has at most two rows, and by using \autoref{prop:reductionprinc} we can reduce it to $\lambda$ being a single row. So the first non-trivial case is when $\mu = (n,k,1)$, which is the subject of the next theorem.
	
	\begin{theorem}\label{thm:nk1}
		Let $n \geq k+1 \geq 3$ and $N = n+k+1$. Then
		\begin{align*}
			Q'_{(n,k,1)}(q) &= q^{k-1}(q^2-t)(q-t)\, Q'_{(N)}(t) \;+\; q^{k-2}(q-t)(q^2+q-t)\, Q'_{(N-1,1)}(t)\\
			&\quad+ \sum_{j=2}^{k-1} q^{k-1-j}(q+1)(q-t)^2\, Q'_{(N-j,j)}(t) \;+\; (q-t)(q-t+1)\, Q'_{(n+1,k)}(t)\\
			&\quad+ (q-t)\, Q'_{(n,k+1)}(t) \;+\; \sum_{j=2}^{k} q^{k-j}(q-t)\, Q'_{(N-j,\,j-1,\,1)}(t) \;+\; Q'_{(n,k,1)}(t).
		\end{align*}
	\end{theorem}
	
	\begin{remark}
		Every coefficient visibly belongs to $\ZZ_{\geq0}[t]$ at $q = t+1$, so \autoref{thm:nk1} proves \autoref{conj:B} for $\mu = (n,k,1)$ with $n \geq k+1 \geq 3$.
	\end{remark}

    By \cite[Ch. III.2, Ex. 2]{MacMainBook}, for a sequence $\alpha = (\alpha_1,\ldots,\alpha_\ell) \in \ZZ^\ell$, if 
    $\alpha_\ell<0$ then $Q'_\alpha(t)=0$, and the raising operator formula \eqref{eq:Qp=hraisingop} implies that $Q'_{(\alpha,0)}(t) = Q'_{\alpha}(t)$. This combined with \autoref{prop:Qpstraightening} gives the following.
    
	% We need two lemmas. 
	% \begin{lemma}\label{lem:vanishing}
	% 	Let $a,b,c \in \ZZ$.
	% 	\begin{enumerate}
	% 		\item If $c < 0$, then $Q'_{(a,b,c)}(t) = 0$.
	% 		\item $Q'_{(a,b,0)}(t) = Q'_{(a,b)}(t)$.
	% 		\item If $b < 0$, then $Q'_{(a,b)}(t) = 0$.
	% 	\end{enumerate}
	% \end{lemma}
	
	% \begin{proof}
	% 	Expand each factor of the product in \eqref{eq:Qp=hraisingop} as
	% 	\begin{align*}
	% 		\dfrac{1-x_i x_j^{-1}}{1-t\, x_i x_j^{-1}} = \sum_{r \geq 0}^{} t^r x_i^r x_j^{-r} \;-\; \sum_{r \geq 0}^{} t^r x_i^{r+1} x_j^{-r-1}.
	% 	\end{align*}
	% 	For the roots $\varepsilon_1 - \varepsilon_3$ and $\varepsilon_2 - \varepsilon_3$, the variable $x_3$ occurs only with non-positive exponents in this expansion, while the root $\varepsilon_1 - \varepsilon_2$ does not involve $x_3$ at all. Hence every monomial in the expansion of $\prod_{\alpha \in R^3_+} \frac{1-x^\alpha}{1-tx^\alpha}\, x^{(a,b,c)}$ has $x_3$-exponent at most $c$. The map $\mathbf{h}$ annihilates any monomial with a negative exponent, so (i) follows. If $c = 0$, the only surviving monomials are those with $x_3$-exponent exactly $0$, i.e., those taking the constant term $1$ from both factors involving $x_3$; what remains is exactly the two-variable expression for $Q'_{(a,b)}(t)$, proving (ii). The same argument with $x_2$ in place of $x_3$ gives (iii).
	% \end{proof}
	
	\begin{lemma}\label{cor:boundarystraight}
		For $m \in \ZZ$,
		\begin{align*}
			&Q'_{(m,0,1)}(t) = t\, Q'_{(m,1)}(t), 
			\\
			&Q'_{(m,-1,1)}(t) = (t-1)\, Q'_{(m)}(t),
			\\
			\hbox{and} \qquad &Q'_{(m,-d,1)}(t) = 0 \quad \hbox{for } d \geq 2.
		\end{align*}
	\end{lemma}
	
	% \begin{proof}
	% 	Apply \autoref{prop:Qpstraightening} to the last two entries $(s,r) = (-d,1)$, so $r-s = d+1$. If $d = 0$ (the case $r-s = 2m+1$ with $m=0$) we get $t\, Q'_{(A,1,0)}(t) = t\, Q'_{(A,1)}(t)$ by \autoref{lem:vanishing}(ii). If $d = 1$ (the case $r-s=2m$ with $m=1$) we get $t\, Q'_{(A,1,-1)}(t) + (t-1)\, Q'_{(A,0,0)}(t) = (t-1)\, Q'_{(A)}(t)$ by \autoref{lem:vanishing}(i),(ii). If $d \geq 2$, then in the notation of \autoref{prop:Qpstraightening}, writing $d+1 = 2m+1$ or $2m$ according to parity, every composition occurring on the right-hand side has last entry at most $\max(-d, -d+m) = -d+m \leq -1$, so all terms vanish by \autoref{lem:vanishing}(i).
	% \end{proof}
	
	\begin{proof}[Proof of \autoref{thm:nk1}]
		By \eqref{eq:Q'KostantPartitionexp},
		\begin{align}\label{eq:nk1kostant}
			Q'_{(n,k,1)}(q) = \sum_{(a,b,c) \in Q^3_+}^{} \cP((a,b,c);q,t)\, Q'_{(n+a,\,k+b,\,1+c)}(t),
		\end{align}
		where $Q^3_+ = \{(a,b,c) \in \ZZ^3 \mid a \geq 0,\, a+b \geq 0,\, a+b+c=0\}$ is the positive part of the root lattice of type $A_2$, and
		\begin{align*}
			\cP((a,b,c);q,t) = \sum_{\kappa}^{} (q-t)^{\#\kappa} q^{|\kappa| - \#\kappa},
		\end{align*}
		the sum running over Kostant partitions $\kappa$ of $(a,b,c)$. The positive roots are $\alpha_1 = \varepsilon_1-\varepsilon_2, \alpha_2 = \varepsilon_2-\varepsilon_3, \alpha_{12} = \varepsilon_1-\varepsilon_3$.
		
		Only the terms with $c \in \{0,-1\}$ survive in \eqref{eq:nk1kostant}. We compute the corresponding Kostant partition polynomials.
		
		\textbf{The case $c = 0$.} Then $b = -a$ and $\nu = a(\varepsilon_1-\varepsilon_2)$ has the unique Kostant partition $\kappa = a\alpha_1+0\alpha_{12}+0\alpha_2$, so, writing $w_a = \cP((a,-a,0);q,t)$,
		\begin{align*}
			w_a = \begin{cases} 
				1, &\hbox{ if } a = 0,
				\\ 
				(q-t)q^{a-1}, &\hbox{ if } a \geq 1.
			\end{cases}
		\end{align*}
		The composition is $(n+a, k-a, 1)$. For $0 \leq a \leq k-1$ it is a partition. For $a \geq k$ we apply \autoref{cor:boundarystraight}: only $a = k$ and $a = k+1$ survive, contributing $t\, Q'_{(n+k,1)}(t)$ and $(t-1)\, Q'_{(n+k+1)}(t)$ respectively. Hence this case contributes
		\begin{align*}
			\sum_{a=0}^{k-1} w_a\, Q'_{(n+a,\,k-a,\,1)}(t) \;+\; t\, w_k\, Q'_{(n+k,1)}(t) \;+\; (t-1)\, w_{k+1}\, Q'_{(n+k+1)}(t).
		\end{align*}
		
		\textbf{The case $c = -1$.} Then $b = 1-a$. The Kostant partitions of $\nu = (a,1-a,-1)$ are $\kappa = a\alpha_1+0\alpha_{12}+1\alpha_2$, and additionally $\kappa = (a-1)\alpha_1+1\alpha_{12}+0\alpha_2$ when $a \geq 1$: solving $\nu = \alpha(\varepsilon_1-\varepsilon_2) + \beta(\varepsilon_1-\varepsilon_3) + \gamma(\varepsilon_2-\varepsilon_3)$ gives $\beta = a-\alpha$, $\gamma = 1-a+\alpha$, which is nonnegative exactly for $\alpha \in \{a-1, a\}$ (and $\alpha = 0$ when $a = 0$). Therefore, writing $\widetilde{w}_a = \cP((a,1-a,-1);q,t)$,
		\begin{align*}
			\widetilde{w}_a = \begin{cases} q-t, &\hbox{ if } a = 0,
				\\ (q-t)^2+(q-t), &\hbox{ if } a = 1,
				\\ (q-t)^2 q^{a-1}+(q-t)^2 q^{a-2}, &\hbox{ if } a \geq 2.\end{cases}
		\end{align*}
		Since $Q'_{(n+a, k+1-a, 0)}(t) = Q'_{(n+a,\,k+1-a)}(t)$, which vanishes for $a \geq k+2$, this case contributes
		\begin{align*}
			\sum_{a=0}^{k+1} \widetilde{w}_a\, Q'_{(n+a,\,k+1-a)}(t).
		\end{align*}
		
		Now collect terms; set $N = n+k+1$ and reindex the two-row terms by the second entry $j = k+1-a$. The three-row terms $Q'_{(n+a,k-a,1)}(t)$ for $0 \leq a \leq k-1$ have coefficients $w_a = q^{k-j}(q-t)$ with $j = k+1-a$ (and $1$ for $a=0$), matching the last line of the claimed formula. We list the coefficients of the remaining terms.
		\begin{enumerate}
			\item[(i)] Coefficient of $Q'_{(n,k+1)}(t)$ ($j = k+1$, $a=0$) is $\widetilde{w}_0 = q-t$.
			\item[(ii)] Coefficient of $Q'_{(n+1,k)}(t)$ ($j = k$, $a=1$) is $\widetilde{w}_1 = (q-t)(q+1-t)$.
			\item[(iii)] Coefficient of $Q'_{(N-j,j)}(t)$ for $2 \leq j \leq k-1$ (i.e., $2 \leq a \leq k-1$) is $\widetilde{w}_a = (q+1)(q-t)^2 q^{a-2} = q^{k-1-j}(q+1)(q-t)^2$.
			\item[(iv)] Coefficient of $Q'_{(N-1,1)}(t)$ ($j=1$, $a=k$) is
			\begin{align*}
				\widetilde{w}_k + t\, w_k = (q+1)(q-t)^2 q^{k-2} + t q^{k-1}(q-t) = q^{k-2}(q-t)(q^2+q-t).
			\end{align*} 
			
			\item[(v)] Coefficient of $Q'_{(N)}(t)$ ($j=0$, $a=k+1$) is
			\begin{align*}
				\widetilde{w}_{k+1} + (t-1)\, w_{k+1} = (q+1)(q-t)^2 q^{k-1} + (t-1)(q-t)q^k = q^{k-1}(q-t)(q^2-t).
			\end{align*}
		\end{enumerate}
		This accounts for all terms of \eqref{eq:nk1kostant}, and the result is exactly the claimed formula.
	\end{proof}
	
	\begin{remark}
		From the theorem, here also we can write \begin{equation}
			b_{\lambda,(n,k,1)}(q,t) = \sum_{T \in \mathrm{SSYT}(\lambda,(n,k,1))} \psi_T(q,t), 
		\end{equation} for $\lambda \vdash n+k+1$ and $n \geq k+1 \geq 3$, where 
		\begin{equation}
			\psi_T(q,t) = \prod_{\substack{x \in \lambda,\ \coarm(x)>0 \\ T(\mathrm{Left}(x)) < T(x)}} \big(q^{\mathrm{ch}_T(x)} - t\big) \prod_{\text{other } x} q^{\mathrm{ch}_T(x)},
		\end{equation}
	\end{remark}
	
	\section{Open questions}\label{sec:open}
	
	We collect the open problems raised in this paper.
	
	\begin{enumerate}
		\item Prove \autoref{conj:B}: $b_{\lambda,\mu}(t+1,t) \in \ZZ_{\geq0}[t]$ for all $\lambda,\mu$. 
		
		\item \label{q:psigeneral} Is there a weight $\psi_T(q,t)$ on semistandard Young tableaux of general shape such that $\sum_{T \in \SSYT(\lambda,\mu)} \psi_T(q,t) = b_{\lambda,\mu}(q,t)$, with $\psi_T(q,0) = q^{\ch(T)}$ and $\psi_T(1,t) = \psi_T(t)$ for each $T$? \autoref{thm:main} solves this for hook shapes and we gather more evidence in \S\ref{sec:smallcases}, \S\ref{sec:nk1}.
		
		\item Alexandersson and Sulzgruber \cite{AlexanderssonSulzgruber} give a combinatorial formula for $b_{\lambda,\mu}(t+1,1)$; generalize it to $b_{\lambda,\mu}(t+n,n)$, which is positive by \autoref{thm:bqtpn}.
		
	\end{enumerate}

	\section*{Acknowledgements}
	 This project was carried out as part of the PACE 2026 workshop at Peking University organized by Yibo Gao. A.B. thanks S. Viswanath and Per Alexandersson for initial discussions regarding \autoref{conj:B}. 
	
	\section*{AI use declaration}
	LLMs were used for finding many of the proofs and editing the paper. ChatGPT 5.6 Sol was used to find the proof of \autoref{prop:stability}, and proofs of \autoref{thm:mainlocal}, \autoref{thm:twocol}, \autoref{thm:nk1} were assisted by Kimi K3. All the results and proofs are verified by the authors and all the explicit coefficients are checked in \texttt{SageMath}.

	\bibliographystyle{alpha}
	\bibliography{main}

\begin{thebibliography}{KRTW26}

\bibitem[Ale21]{Alexandersson_LLTelementarymelting}
Per Alexandersson.
\newblock {LLT} polynomials, elementary symmetric functions and melting
  lollipops.
\newblock {\em J. Algebr. Comb.}, 53(2):299--325, 2021.

\bibitem[AS22]{AlexanderssonSulzgruber}
Per Alexandersson and Robin Sulzgruber.
\newblock A combinatorial expansion of vertical-strip {LLT} polynomials in the
  basis of elementary symmetric functions.
\newblock {\em Adv. Math.}, 400:58, 2022.
\newblock Id/No 108256.

\bibitem[Bha23]{Bhattacharya22}
Aritra Bhattacharya.
\newblock Haglund's positivity conjecture for multiplicity one pairs.
\newblock {\em S{\'e}min. Lothar. Comb.}, 87:article b87b, 20, 2023.

\bibitem[Bha24a]{Bhattacharya_Thesis}
Aritra Bhattacharya.
\newblock {\em Haglund's Conjecture and Clebsch--Gordan Rule for Macdonald
  Polynomials}.
\newblock PhD thesis, Homi Bhabha National Institute, Chennai, India, April
  2024.
\newblock IMSc Theses/Dissertations, HBNI Th241.

\bibitem[Bha24b]{Bhattacharya2024monomialexpansionformulahalllittlewood}
Aritra Bhattacharya.
\newblock The monomial expansion formula for {H}all-{L}ittlewood
  {$P$}-polynomials, \\ \url{https://arxiv.org/abs/2407.14652}, 2024.

\bibitem[BZ90]{BZ90}
A.~D. Berenshte\u{\i}n and A.~V. Zelevinski\u{\i}.
\newblock When is the multiplicity of a weight equal to {$1$}?
\newblock {\em Funktsional. Anal. i Prilozhen.}, 24(4):1--13, 96, 1990.

\bibitem[Car98]{Carbonara_inverseKFp}
Joaquin~O. Carbonara.
\newblock A combinatorial interpretation of the inverse {$t$}-{K}ostka matrix.
\newblock volume 193, pages 117--145. 1998.
\newblock Selected papers in honor of Adriano Garsia (Taormina, 1994).

\bibitem[Hai99]{HaimanGeometry}
Mark Haiman.
\newblock Macdonald polynomials and geometry.
\newblock In {\em New perspectives in algebraic combinatorics}, pages 207--254.
  Cambridge: Cambridge University Press, 1999.

\bibitem[KRTW26]{KRTW_remarks}
Shashank Kanade, Matthew~C. Russell, Shunsuke Tsuchioka, and S.~Ole Warnaar.
\newblock Remarks on the conjectures of {Capparelli}, {Meurman}, {Primc} and
  {Primc}.
\newblock Preprint, {arXiv}:2404.03851 [math.{CO}] (2026), 2026.

\bibitem[Las05]{Lascoux_addingpm1}
Alain Lascoux.
\newblock Adding {{\(\pm 1\)}} to the argument of a {Hall}-{Littlewood}
  polynomial.
\newblock {\em S{\'e}min. Lothar. Comb.}, 54:b54n, 17, 2005.

\bibitem[Lot11]{Lothaire_AConwords}
M.~Lothaire.
\newblock {\em Algebraic combinatorics on words. {With} preface by {Jean}
  {Berstel} and {Dominique} {Perrin}.}, volume~90 of {\em Encycl. Math. Appl.}
\newblock Cambridge: Cambridge University Press, reprint of the 2002 hardback
  ed. edition, 2011.

\bibitem[LS78]{LS_suruneFoulkes}
Alain Lascoux and Marcel-Paul Sch{\"u}tzenberger.
\newblock Sur une conjecture de {H}. {O}. {Foulkes}.
\newblock {\em C. R. Acad. Sci., Paris, S{\'e}r. A}, 286:323--324, 1978.

\bibitem[Mac95]{MacMainBook}
I.~G. Macdonald.
\newblock {\em Symmetric functions and {H}all polynomials}.
\newblock Oxford Mathematical Monographs. The Clarendon Press, Oxford
  University Press, New York, second edition, 1995.
\newblock With contributions by A. Zelevinsky, Oxford Science Publications.

\bibitem[Mak]{MO190646}
Igor Makhlin.
\newblock On a positivity property of {H}all--{L}ittlewood polynomials.
\newblock MathOverflow.
\newblock URL:https://mathoverflow.net/q/190646 (version: 2014-12-17).

\bibitem[Mel20]{Mellit_Poincarepols}
Anton Mellit.
\newblock Poincar{\'e} polynomials of character varieties, {Macdonald}
  polynomials and affine {Springer} fibers.
\newblock {\em Ann. Math. (2)}, 192(1):165--228, 2020.

\bibitem[Sta99]{EC2}
Richard~P. Stanley.
\newblock {\em Enumerative combinatorics. {Volume} 2}, volume~62 of {\em Camb.
  Stud. Adv. Math.}
\newblock Cambridge: Cambridge University Press, 1999.

\bibitem[Zab98]{Zabrocki_thesis}
Michael Zabrocki.
\newblock {\em On the Action of the {Hall}--{Littlewood} Vertex Operator}.
\newblock PhD thesis, University of California, San Diego, 1998.
\newblock Ph.D. dissertation.

\end{thebibliography}

\end{document}